\documentclass[11pt,reqno]{amsart}
\usepackage[utf8]{inputenc}
\usepackage[T1]{fontenc} 
\usepackage{lmodern}
\usepackage{amsfonts,amsthm,amsmath,amssymb,thmtools}
\usepackage{booktabs}
\usepackage{boldline,makecell}
\usepackage{array,diagbox}
\usepackage{graphicx}
\usepackage[dvipsnames,table]{xcolor}   
\usepackage{enumerate}
\usepackage{hyperref}
\usepackage{color}
\usepackage{tikz-cd}
\usetikzlibrary{shapes.misc}

\usepackage{transparent}
\usepackage[margin=1in]{geometry}
\usepackage[
    maxbibnames=99,
    backend=biber,
    style=alphabetic,
    sorting=nyt,
    giveninits=true,
    url=false,
    doi=false,
    isbn=false
]{biblatex}
\DeclareFieldFormat{pages}{#1}
\renewbibmacro{in:}{%
  \ifentrytype{article}
    {}
    {\bibstring{in}%
     \printunit{\intitlepunct}}}
\DeclareFieldFormat
[article,inbook,incollection,inproceedings,patent,thesis,unpublished]
  {title}{\mkbibemph{#1}}
\DeclareFieldFormat{journaltitle}{#1\isdot}
\DeclareFieldFormat[article]{volume}{\mkbibbold{#1}}
\DeclareFieldFormat[article]{number}{\bibstring{number}\addnbspace #1}

\renewbibmacro*{journal+issuetitle}{%
  \usebibmacro{journal}%
  \setunit*{\addspace}%
  \iffieldundef{series}
    {}
    {\newunit
     \printfield{series}%
     \setunit{\addspace}}%
  \printfield{volume}%
  \setunit{\addspace}%
  \usebibmacro{issue+date}%
  \setunit{\addcomma\space}%
  \printfield{number}%
  \setunit{\addcolon\space}%
  \usebibmacro{issue}%
  \setunit{\addcomma\space}%
  \printfield{eid}
  \newunit}
\DeclareFieldFormat{series}{#1\isdot}

\newtheoremstyle{mystyle}
  {10pt}
  {4pt}
  {\itshape}
  {}
  {\bfseries}
  {.}
  { }
  {\thmname{#1}\thmnumber{ #2}\thmnote{ (#3)}}

\theoremstyle{mystyle}
\newtheorem{Thm}{Theorem}[section]
\newtheorem{Lem}[Thm]{Lemma}
\newtheorem{Cor}[Thm]{Corollary}
\newtheorem{Prop}[Thm]{Proposition}

\newtheorem{Que}[Thm]{Question}

\theoremstyle{definition}

\newtheorem{Ex}[Thm]{Example}

\theoremstyle{remark}
\newtheorem{Rmk}[Thm]{Remark}

\declaretheoremstyle[%
  spaceabove=3pt,
  spacebelow=10pt,
  headfont=\normalfont\itshape,%
  postheadspace=.5em,%
  qed=\qedsymbol%
]{mystyle2}

\newcommand{\R}{\mathbb{R}}
\newcommand{\Z}{\mathbb{Z}}

\author{Ian Agol}
\address{Department of Mathematics, University of California, Berkeley, Berkeley, CA 94720, USA}
\email{ianagol@berkeley.edu}

\author{Qiuyu Ren}
\address{Department of Mathematics, University of California, Berkeley, Berkeley, CA 94720, USA}
\email{qiuyu\_ren@berkeley.edu}

\title{Ribbon concordance of fibered knots and compressions of surface homeomorphisms}

\begin{document}

\begin{abstract}
We prove that simplicial volume and dilatation are monotone under ribbon concordance between fibered knots in $S^3$, and that every fibered knot has only finitely many predecessors in the ribbon-concordance partial order, providing evidence for questions raised by Gordon. We also give an algorithm to enumerate, up to symmetries, all minimal compressions of a surface homeomorphism, extending a theorem of Casson--Long. This yields an algorithm to find all knots that are strongly homotopy-ribbon concordant to a given fibered knot in some homotopy $I\times S^3$. Our study of minimal compressions also provides an alternative perspective on results of Miyazaki concerning nonsimple fibered ribbon knots.
\end{abstract}

\maketitle

\section{Introduction}
\subsection{Ribbon concordances of fibered knots}\label{sbsec:intro_ribbon}
A (smooth) concordance $C\subset I\times S^3$ from a knot $J$ to a knot $K$ is a (smoothly) embedded annulus cobounding $J$ and $K$. It is a \textit{ribbon concordance} if the projection $I\times S^3\to I$ restricts to a Morse function on $C$ with no index-$2$ critical points. We say that $J$ is \textit{ribbon concordant} to $K$, denoted $J\le K$, if there exists such a concordance.

The notion of ribbon concordance was introduced by Gordon \cite{gordon1981ribbon} and has been studied extensively in recent years. If $J\le K$, then $J$ is expected to be ``simpler'' than $K$. In particular, many knot invariants are known to behave monotonically under ribbon concordance, including the Alexander polynomial and the $S$-equivalence class \cite{gilmer1984ribbon}, Seifert genus and knot Floer homology \cite{zemke2019knot}, and Khovanov homology \cite{levine2019khovanov}. The first author recently proved that ribbon concordance is a partial order \cite{agol2022ribbon}, confirming a conjecture of Gordon \cite{gordon1981ribbon}.

Our work is motivated by the following questions of Gordon \cite{gordon1981ribbon}:
\begin{Que}[{\cite[Question~6.4]{gordon1981ribbon}}]
Does $J\le K$ imply $||S^3\backslash J||\le||S^3\backslash K||$?
\end{Que}
Here, $||\cdot||$ denotes the simplicial volume. 

\begin{Que}[{\cite[Question~6.2]{gordon1981ribbon}}]\label{que:stabilize}
If $K_1\ge K_2\ge\cdots$, does there exist $m$ such that $K_n=K_m$ for all $n\ge m$?
\end{Que}

The following refinement of Question~\ref{que:stabilize} was raised by Baldwin--Sivek \cite{baldwin2025ribbon}.
\begin{Que}[{\cite[Question~1.1]{baldwin2025ribbon}}]
For each knot $K$, are there only finitely many $J$ with $J\le K$?
\end{Que}

We answer these questions affirmatively when $K$ is fibered. Note that if $K$ is fibered and $J\le K$, then $J$ is also fibered \cite{silver1992knot,kochloukova2006some}.

\begin{Thm}\label{thm:volume}
If $K$ is fibered and $J\le K$, then $||S^3\backslash J||\le||S^3\backslash K||$.
\end{Thm}

For a fibered knot $K$, the monodromy $\phi_K\colon S\to S$ is a surface homeomorphism. By the Nielsen--Thurston classification, $\phi_K$ can be canonically decomposed along a multicurve into periodic and pseudo-Anosov pieces. The \textit{dilatation} of $K$, denoted $\lambda(K)$, is the maximal dilatation constant among pseudo-Anosov pieces of $\phi_K$ (or $1$, if there is no pseudo-Anosov piece).

\begin{Thm}\label{thm:dilatation}
If $K$ is fibered and $J\le K$, then $\lambda(J)\le\lambda(K)$.
\end{Thm}

The following theorem is deduced as a consequence of Theorem~\ref{thm:dilatation}.
\begin{Thm}\label{thm:finite}
If $K$ is fibered, then there are finitely many $J$ with $J\le K$.
\end{Thm}

Results analogous to Theorems~\ref{thm:volume},~\ref{thm:dilatation} were established independently by Baldwin--Sivek \cite[Theorem~1.5,Lemma~2.6]{baldwin2025ribbon} using Floer-theoretic techniques. Our proofs are purely topological and give sharper bounds between fibered knots. Theorem~\ref{thm:finite} was independently obtained by Baldwin--Hanselman--Sivek \cite[Corollary~1.3]{baldwin2026ribbon} as a consequence of the stronger result that every knot has at most finitely many ribbon predecessors that are fibered, via more intricate Floer-theoretic arguments.

\subsection{Compressions of surface homeomorphisms}
Our proofs of Theorem~\ref{thm:volume},~\ref{thm:dilatation},~\ref{thm:finite} start with the following consequence of Casson--Gordon \cite{casson1983loop}.

\begin{Thm}[\cite{casson1983loop}]\label{thm:CG}
If $J,K$ are fibered, then $J$ admits a strongly homotopy-ribbon concordance to $K$ in some homotopy $I\times S^3$ if and only if the monodromy of $K$ compresses to that of $J$.
\end{Thm}
See Section~\ref{sec:prelim} for the definition of strongly homotopy-ribbon concordances and of compressions of surface homeomorphisms. We write $J\le_hK$ if $J$ admits a strongly homotopy-ribbon concordance to $K$ in some homotopy $I\times S^3$. Every ribbon concordance is a strongly homotopy-ribbon concordance, and in fact the statements of Theorem~\ref{thm:volume},~\ref{thm:dilatation},~\ref{thm:finite} can be strengthened by replacing $\le$ with $\le_h$. To our knowledge, it is open whether $\le_h$ is a strictly finer relation than $\le$.

We now explain a second approach to proving Theorem~\ref{thm:finite}, which leads to some results of independent interest.

Let $S$ be a compact oriented surface with (possibly empty) boundary and $\phi\colon S\to S$ be an orientation-preserving surface homeomorphism (not necessarily rel boundary). Motivated by the study of cobordisms of surface homeomorphisms \cite{bonahon1983cobordism} as well as of fibered ribbon knots \cite{casson1983loop}, Casson--Long \cite{casson1985algorithmic} proved the following results.
\begin{Thm}[{\cite{casson1985algorithmic}}]\label{thm:CL}
Let $\phi\colon S\to S$ be a surface homeomorphism as above.
\begin{enumerate}[(1)]
\setcounter{enumi}{-1}
\item There is an algorithm to determine whether $\phi$ admits a nontrivial compression.
\item If $\phi$ is pseudo-Anosov, then it admits at most finitely many minimal compressions.
\item If $\phi$ is pseudo-Anosov, there is an algorithm to find all minimal compressions of $\phi$.
\end{enumerate}
\end{Thm}
The pseudo-Anosov assumption in Theorem~\ref{thm:CL} is necessary as stated. However, by taking additional symmetries into account, we are able to remove this assumption and obtain the following generalization of Casson--Long's main results.

\begin{Thm}\label{thm:compression}
Let $\phi\colon S\to S$ be a surface homeomorphism as described above.
\begin{enumerate}[(1)]
\item $\phi$ admits at most finitely many minimal compressions up to symmetries of the pair $(S,\phi)$.
\item There is an algorithm to find all minimal compressions of $\phi$ up to symmetries of $(S,\phi)$.
\end{enumerate}
\end{Thm}

See Section~\ref{sbsec:comp_ex} for a precise explanation of the minimality condition in Theorem~\ref{thm:CL} and Theorem~\ref{thm:compression} and of the symmetries involved in Theorem~\ref{thm:compression}, as well as some examples. It suffices to say here that minimality means compressing by the least possible amount.

One cannot iterate Theorem~\ref{thm:compression} itself to deduce a version without the minimality assumption (which is false), since the relevant symmetries may not extend. Nevertheless, finiteness remains valid if one only remembers the interior boundary of the compressions.
\begin{Cor}\label{cor:compression}
Every surface homeomorphism $\phi$ of a compact oriented surface compresses to only finitely many surface homeomorphisms $\phi'$ (up to isotopy and conjugation). Moreover, there is an algorithm to find all such $\phi'$.\qed
\end{Cor}
We deduce Theorem~\ref{thm:finite} as a corollary of Theorem~\ref{thm:CG} and Corollary~\ref{cor:compression}.

\begin{Cor}\label{cor:find_predecessors}
For any given fibered knot $K$, there are only finitely many knots $J$ with $J\le_hK$. Moreover, there is an algorithm to find all such $J$.
\end{Cor}

\begin{Rmk}
Even if one is only interested in pseudo-Anosov homeomorphisms, it appears necessary to study minimal compressions of periodic and reducible homeomorphisms in order to obtain Corollary~\ref{cor:compression} and Corollary~\ref{cor:find_predecessors}.
\end{Rmk}

Corollary~\ref{cor:find_predecessors} in particular gives an algorithm to determine whether a fibered knot is strongly homotopy-ribbon in a homotopy $B^4$. Modulo the slice-ribbon conjecture and the smooth $4$-dimensional Poincar\'e conjecture, this can in principle be used to determine the smoothly slice status of any fibered knot, in particular those in \cite[Table~5]{dunfield2025ribbon} (the implication may also go in other directions). In practice, the most interesting part of the algorithm comes from that of Casson--Long \cite{casson1985algorithmic}.\smallskip

Since Khovanov homology is only known to obstruct strongly homotopy-ribbon concordance between knots in the standard $I\times S^3$ \cite[Theorem~1.4]{gujral2022khovanov}, one may hope to find an exotic $I\times S^3$, hence an exotic $4$-sphere, by combining the algorithm in Corollary~\ref{cor:find_predecessors} with the obstruction from Khovanov homology.\medskip

Finally, our study of minimal compressions of surface homeomorphisms provides a conceptually simpler perspective on results of Miyazaki \cite{miyazaki1994nonsimple} on nonsimple (i.e. satellite) fibered ribbon knots. Rather than attempting a comprehensive treatment, we illustrate the idea by showing that the $(2,1)$-cable of the figure $8$ knot is not strongly homotopy-ribbon, and by giving simple proofs of the following results in the spirit of Miyazaki.

\begin{Thm}\label{thm:miyazaki}
Let $J_1,\cdots,J_m,K_1,\cdots,K_n$ be prime fibered knots, $m,n\ge0$. Then $J_1\#\cdots\#J_m\le_hK_1\#\cdots\#K_n$ if and only if there exist knots $K_{i,j}$, $1\le j\le\ell_i$, $1\le i\le n$, for some $\ell_i\ge0$, such that
\begin{itemize}
\item $K_{i,1}\#\cdots\#K_{i,\ell_i}\le_hK_i$ for all $i$;
\item the collection of knots $K_{i,j}$ consists of exactly the knots $J_1,\cdots,J_m$ and some other $2k$ knots $J_{m+1},-J_{m+1},\cdots,J_{m+k},-J_{m+k}$ that come in (reversed) mirrored pairs, $k\ge0$.
\end{itemize}
\end{Thm}

\begin{Cor}\label{cor:miyazaki}
\begin{enumerate}
\item If the slice-ribbon conjecture is true, then every concordance class of knots contains at most one fibered knot that is minimal with respect to $\le_h$.
\item If the slice-ribbon conjecture and the smooth $4$-dimensional Poincar\'e conjecture are both true, then no torsion element in the knot concordance group of order greater than $2$ can be represented by a fibered knot.
\end{enumerate}
\end{Cor}
Theorem~\ref{thm:miyazaki} implies parts of Theorems~5.3 and 5.5 in \cite{miyazaki1994nonsimple}, and conversely, Miyazaki's proof of these results may also apply to show Theorem~\ref{thm:miyazaki}. Corollary~\ref{cor:miyazaki}(1) is an immediate consequence of Miyazaki (see \cite[Remark~6]{baker2016note}), and Corollary~\ref{cor:miyazaki}(2) follows from Theorem~5.8 in \cite{miyazaki1994nonsimple}. We remark that it is open whether every concordance class of knots has a unique minimum with respect to $\le$ (or $\le_h$), and whether the knot concordance group contains a torsion element of order greater than $2$.

We end the introduction with the following question.
\begin{Que}\label{que:fibered_conc}
If $K_1$ and $K_2$ are concordant fibered knots, must there be a fibered knot $K$ with $K_1\le K$, $K_2\le K$? (Or with $\le$ replaced by $\le_h$?)
\end{Que}

One may replace the existence of $K$ in Question~\ref{que:fibered_conc} by the equivalent condition that $K_1,K_2$ are related by a zigzag of (strongly homotopy-)ribbon concordances between fibered knots. Note that Question~\ref{que:fibered_conc} is implied by the slice-ribbon conjecture, since one could take $K=K_1\#K_2\#(-K_1)$.

Using the characteristic submanifold theory in the spirit of the arguments in \cite{bonahon1983cobordism}, one can show that if $K_1,K_2$ are hyperbolic knots with genus at most $3$, and are each minimal with respect to $\le_h$, then the existence of a fibered knot $K$ with $K_1\le_hK$, $K_2\le_hK$ implies that $K_1=K_2$.

\subsection{Organization of the paper}
In Section~\ref{sec:prelim}, we give some preliminaries and prove Theorem~\ref{thm:CG}. In Section~\ref{sec:example}, we give examples of ribbon concordances between fibered knots. We prove Theorem~\ref{thm:volume} in Section~\ref{sec:volume}, Theorem~\ref{thm:dilatation} and Theorem~\ref{thm:finite} in Section~\ref{sec:dilatation}, and Theorem~\ref{thm:compression} in Section~\ref{sec:compression}. In Section~\ref{sbsec:comp_forms}, we classify minimal compressions of surface homeomorphisms into six canonical forms. Based on this classification, we prove Theorem~\ref{thm:miyazaki} and Corollary~\ref{cor:miyazaki} in Section~\ref{sec:nonsimple}.

\subsubsection*{Acknowledgments}
We thank John Baldwin, Nathan Dunfield and Chi Cheuk Tsang for helpful discussions. QR acknowledges the use of ChatGPT and Gemini for providing some figures and partial arguments when preparing this manuscript. The authors were partially supported by the Simons Investigator Award 376200.

\section{Preliminaries}\label{sec:prelim}
In this section, we explain Theorem~\ref{thm:CG}, which reduces our proofs of Theorems~\ref{thm:volume} and Theorem~\ref{thm:dilatation} to the study of compressions of surface homeomorphisms. Throughout this paper, unless explicitly stated otherwise, all surfaces are compact, oriented, and all surface homeomorphisms are orientation-preserving, but not necessarily rel boundary.

A concordance from $J$ to $K$ in a homotopy $I\times S^3$ is \textit{strongly homotopy-ribbon} if its complement admits a relative handle decomposition with only $1$- and $2$-handles. If such a concordance exists, we write $J\le_hK$. The arguments in \cite{agol2022ribbon} show that $\le_h$ is a partial order.

A \textit{(relative) compression body} is a compact oriented $3$-manifold $C$ whose boundary is partitioned as $\partial C=(-\partial_iC)\cup A\cup\partial_eC$, so that
\begin{itemize}
\item $\partial_iC$ and $\partial_eC$, called the \textit{interior boundary} and the \textit{exterior boundary} of $C$, are disjoint subsurfaces of $\partial C$.
\item $\partial_iC$ has no sphere components.
\item $A\cong I\times\partial(\partial_iC)\cong I\times\partial(\partial_eC)$, called the \textit{vertical boundary} of $C$, is a union of annuli, each component of which cobounds a component of $\partial(\partial_iC)$ and a component of $\partial(\partial_eC)$.
\item $C$ is obtained from $I\times\partial_iC$ by attaching $0$-handles and $1$-handles; equivalently, $C$ is obtained from (a thickening of) $\partial_eC$ by attaching $2$-handles and $3$-handles.
\end{itemize}
For example, every handlebody is a compression body with empty interior boundary and connected exterior boundary. In general, one may regard a compression body $C$ as a cobordism from the interior boundary to the exterior boundary.

When $C$ is connected, each component of $\partial_iC$ is $\pi_1$-injective in $C$, and the unique component of $\partial_eC$ is $\pi_1$-surjective onto $C$. Conversely, if $C$ is a connected oriented irreducible $3$-manifold and $S\subset\partial C$ is a connected subsurface that $\pi_1$-surjects onto $C$, then by applying the loop theorem iteratively and the Poincar\'e conjecture, we know that $C$ is a relative compression body with exterior boundary $S$.

Let $\phi\colon S\to S$ be a surface homeomorphism. If $C$ is a compression body with $\partial_eC=S$ and $\Phi\colon C\to C$ is a homeomorphism extending $\phi$, we say $\phi$ \textit{compresses} in $C$ to the surface homeomorphism $\Phi|_{\partial_iC}\colon\partial_iC\to\partial_iC$, and we call $\Phi\colon C\to C$, or simply $C$, a \textit{compression} of $\phi$. Note that $\Phi$ is determined by $C$ up to isotopy rel $S$. We do not regard different extensions $\Phi$ of $\phi$ over $C$ to be different compressions; consequently, $\Phi|_{\partial_iC}$ is only well-defined up to isotopy. Similarly, if $\phi,\phi'\colon S\to S$ are isotopic, then their compressions are in one-to-one correspondence; hence, one may speak of compressions of a mapping class (not rel boundary) rather than a concrete surface homeomorphism, and we often confuse these two notions henceforth.

When $\phi$ is a rel boundary surface homeomorphism, e.g. the monodromy of a fibered knot, we require a compression $\Phi\colon C\to C$ of $\phi$ to be the identity on the vertical boundary. The distinction is minor, and starting from the next section, we ignore the difference and only talk about compressions of surface homeomorphisms not necessarily rel boundary.

\begin{proof}[Proof of Theorem~\ref{thm:CG}]
Necessity: Let $C$ be a strongly homotopy-ribbon concordance between $J$ and $K$ in a homotopy $I\times S^3$. After deleting the tubular neighborhood of an arc on $C$ connecting $J$ to $K$, the rest of $C$, denoted $D$, is a ribbon disk of $(-J)\#K$ in a homotopy $4$-ball $B$. The knot $(-J)\#K$ is a fibered knot with closed monodromy $(-\phi_J)\cup\phi_K\colon(-F_J)\cup F_K\to(-F_J)\cup F_K$, where $\phi_\bullet\colon F_\bullet\to F_\bullet$ is the monodromy of $\bullet=J,K$. Write for short $\phi=(-\phi_J)\cup\phi_K$ and $F=(-F_J)\cup F_K$. Casson--Gordon \cite[Theorem~5.1]{casson1983loop} shows that the monodromy $\phi$ compresses in a handlebody $H$ with $\ker(\pi_1F\to\pi_1H)=\ker(\pi_1F\to\pi_1W/(\pi_1W)_\omega)$ where $W$ is the infinite cyclic cover of $B\backslash D$ and $(\pi_1W)_\omega$ is a certain characteristic subgroup of $\pi_1W$ (in fact, it is the trivial group since $\pi_1W$ is free by \cite{kochloukova2006some}, although we won't need this). We claim that $\pi_1F_K\to\pi_1H$ is surjective. This would imply that $H$ is a compression body with interior boundary $F_J$ and exterior boundary $F_K$, proving the desired statement.

Since $C$ is a strongly homotopy-ribbon concordance, $B\backslash D$ can be built from $S^3\backslash K$ by attaching $2,3$-handles, showing that $\pi_1(S^3\backslash K)\to\pi_1(B\backslash D)$ is surjective. Taking the infinite cyclic cover, we find that $\pi_1(\R\times F_K)\to\pi_1W$ is surjective. Since $\pi_1H=\pi_1W/(\pi_1W)_\omega$, this proves the claim.

Sufficiency: Let $C$ be a compression body and $\Phi\colon C\to C$ a homeomorphism that restricts to monodromies of $J$ and $K$ on the interior and exterior boundaries, respectively. The vertical boundary of the mapping torus $M_\Phi$ of $\Phi$ is $I\times(S^1\times S^1)$. Attaching to the vertical boundary of $M_\Phi$ a copy of $I\times(S^1\times D^2)$ that kills the meridional slopes on both $K$ and $J$ (this is possible because $\Phi$ is the identity on the vertical boundary of $C$), we obtain a homotopy $I\times S^3$, inside which the core $I\times S^1\times\{*\}$ of $I\times S^1\times D^2$ is a concordance between $J$ and $K$, whose complement, namely $M_\Phi$, is built from $S^3\backslash J$ by attaching $1,2$-handles (since $C$ is built from $\partial_iC$ by attaching $1$-handles).
\end{proof}

We note that if $K$ is a fibered knot whose monodromy compresses to some homeomorphism $\phi\colon S\to S$ of a connected surface, then $\phi=\phi_J$ is the monodromy of a fibered knot $J$ (this is \cite[Lemma~1.3]{miyazaki1994nonsimple}(2) after geometrization) with $J\le_hK$. Hence, Corollary~\ref{cor:find_predecessors} is a direct consequence of Theorem~\ref{thm:CG} and Corollary~\ref{cor:compression}.

\section{Examples}\label{sec:example}
Given a fibered knot $J$, it is natural to wonder what fibered knots $K$ with $J\leq K$ might look like. If $R$ is a fibered ribbon knot, then one can see that $J \leq J\# R$. In this case, the theorems of this paper about monotonicity of simplicial volume and dilatation are easy to verify since $J$ is a retract of $J\#R$.

A variation on this is to take a fibered ribbon knot $R$, such as the square knot, such that a ribbon disk of it shares a $2n$-gon $Q$ with its fiber surface. Pick any $2n$-gon $P$ on the fiber surface of a fibered knot $J$. Then we may take the Murasugi sum of $J$ and $R$ along $P$ and $Q$ to obtain a knot $K$, which is again fibered \cite[Theorem~1.3]{gabai1983murasugi}. Pushing $J$ slightly into the ribbon surface, we get a ribbon annulus between $J$ and $K$, showing that $J\le K$. See Figure~\ref{fig:example}.

\begin{figure}[htbp]
\centering
\def\svgwidth{.9\columnwidth}
\begingroup%
  \makeatletter%
  \providecommand\color[2][]{%
    \errmessage{(Inkscape) Color is used for the text in Inkscape, but the package 'color.sty' is not loaded}%
    \renewcommand\color[2][]{}%
  }%
  \providecommand\transparent[1]{%
    \errmessage{(Inkscape) Transparency is used (non-zero) for the text in Inkscape, but the package 'transparent.sty' is not loaded}%
    \renewcommand\transparent[1]{}%
  }%
  \providecommand\rotatebox[2]{#2}%
  \newcommand*\fsize{\dimexpr\f@size pt\relax}%
  \newcommand*\lineheight[1]{\fontsize{\fsize}{#1\fsize}\selectfont}%
  \ifx\svgwidth\undefined%
    \setlength{\unitlength}{544.62837672bp}%
    \ifx\svgscale\undefined%
      \relax%
    \else%
      \setlength{\unitlength}{\unitlength * \real{\svgscale}}%
    \fi%
  \else%
    \setlength{\unitlength}{\svgwidth}%
  \fi%
  \global\let\svgwidth\undefined%
  \global\let\svgscale\undefined%
  \makeatother%
  \begin{picture}(1,0.30287918)%
    \lineheight{1}%
    \setlength\tabcolsep{0pt}%
    \put(0,0){\includegraphics[width=\unitlength,page=1]{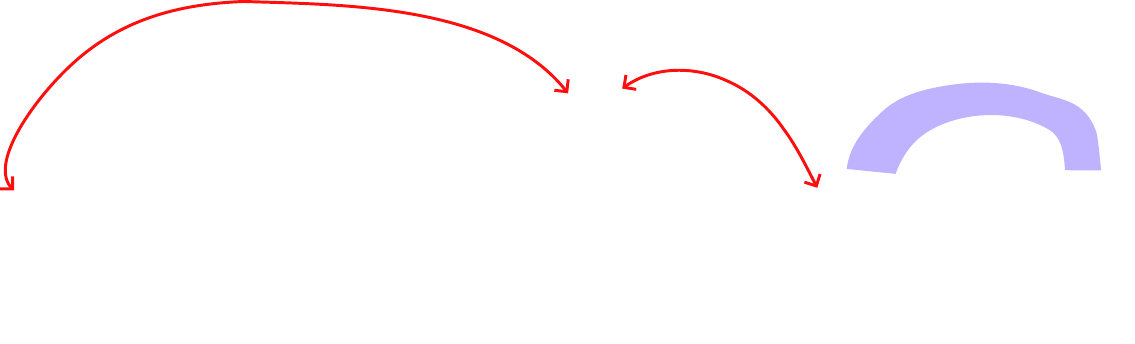}}%
    \put(0.48700724,0.26358119){\color[rgb]{1,0,0}\transparent{0.94599998}\makebox(0,0)[lt]{\lineheight{1.25}\smash{\begin{tabular}[t]{l}plumb\end{tabular}}}}%
    \put(0,0){\includegraphics[width=\unitlength,page=2]{plumb.pdf}}%
  \end{picture}%
\endgroup%

\caption{Plumb the square knot onto the figure $8$ knot, yielding a ribbon concordance from the figure $8$ knot (to the knot $9_{24}$ or its mirror).}
\label{fig:example}
\end{figure}

\section{Monotonicity of simplicial volume}\label{sec:volume}

To prove monotonicity of simplicial volume, we use the dual relation with bounded cohomology. First we review the relation between the Gromov norm on homology and bounded cohomology \cite{Gromov1982volume}. We will follow the notation of \cite[Appendix F]{Benedetti1992Petronio}. 

Let $X$ be a topological space. For $c\in C^k(X)$, define $||c||_\infty= \sup \{ |c(\sigma)| \mid \sigma\colon \Delta_k \to X\}  \in [0,\infty]$.
Define $\hat{C}^k(X) = \{c\in C^k(X) \mid ||c||_\infty < \infty \}$. Then $\delta^k(\hat{C}^k(X))\subset \hat{C}^{k+1}(X)$, and the \textit{bounded cohomology} of $X$, denoted $\hat{H}^*(X)$, is the cohomology
of the cochain complex $(\hat{C}^*(X), \delta^*)$. The $\ell^\infty$-norm on $\hat C^*(X)$ induces a pseudo-norm on $\hat H^*(X)$, namely $||\beta||_\infty = \inf \{ ||c||_\infty \mid c\in \hat{C}^k(X), \langle c \rangle = \beta\}$, $\beta\in \hat{H}^k(X)$.
The cohomology groups $\hat{H}^\ast(X)$ satisfy naturality, homotopy, and dimension axioms, but not the other axioms of cohomology theory. 
Moreover, they only depend on $\pi_1(X)$. The relation with Gromov norm is as follows. 

\begin{Prop}\label{prop:l1_linfty}\cite[Prop. F.2.2]{Benedetti1992Petronio}
The Gromov norm $||z||_1$ of a class $z\in H_k(X)$ satisfies $||z||_1^{-1} = \inf \{||\beta||_\infty \mid \beta\in \hat{H}^k(X), [\beta,z]=1\}$.\qed
\end{Prop}

Now we state the theorem which will imply the monotonicity of simplicial volume for ribbon concordances of fibered knots. 

\begin{Thm}\label{thm:Gromov_norm_compressionbody}
    Let $C$ be a connected compression body, and let $\Phi\colon C\to C$ be a homeomorphism that restricts to $\Phi_i = \Phi|_{\partial_i C}, \Phi_e= \Phi|_{\partial_e C}$.
    Then $||T_{\Phi_i}|| \leq ||T_{\Phi_e}||$.
\end{Thm}

\begin{proof}
The simplicial volume of a double is twice the simplicial volume. Thus, to avoid dealing with manifolds with boundary, we may double the compression body along the vertical boundary, and assume that $\partial_eC$ and $\partial_iC$ are closed surfaces.

Let $[T_{\Phi_e}], [T_{\Phi_i}]$ be the fundamental classes of the respective manifolds. Then $[T_{\Phi_e}], [T_{\Phi_i}]$ both map to a common generator of $H_3(T_\Phi)\cong \mathbb{Z}$, denoted by $[T_\Phi]$. We have $||T_{\Phi_e}||=||[T_{\Phi_e}]||_1\geq ||[T_\Phi]||_1$ since maps are non-increasing on Gromov norm. We also have $||T_{\Phi_i}||\geq ||[T_\Phi]||_1$, and if we can prove the other inequality, then the theorem will follow. 

By a theorem of Gromov \cite[p. 247]{Gromov1982volume}, the cyclic covering $\rho\colon \mathbb{R}\times C\to T_\Phi$ induces an isometric injection $\hat{\rho}^\ast\colon \hat{H}^3(T_\Phi)\to \hat{H}^3(\mathbb{R}\times C) \cong \hat{H}^3(C)$ with respect to $||\cdot||_\infty$, and similarly for $T_{\Phi_e},T_{\Phi_i}$, since $\mathbb{Z}$ is amenable.
 
We need to know naturality in Gromov's theorem. The proof of Gromov's theorem uses an averaging operator $\hat{A}\colon \hat{H}^\ast(C)\to \hat{H}^\ast(T_\Phi)$ such that $\hat{A}\circ \hat{\rho}^\ast= \mathrm{id}$, and similarly replacing $\Phi$ with $\Phi_e,\Phi_i$. Let $\iota\colon T_{\Phi_i}\to T_\Phi$ denote the inclusion map, inducing the inclusion of cyclic covers $\iota\colon \mathbb{R}\times \partial_i C \to \mathbb{R}\times C$. Then one can easily see from Gromov's proof that $\hat{A}\circ\hat{\iota}^\ast = \hat{\iota}^\ast\circ\hat{A}$. The point is that $\hat{A}$ is defined via an averaging operator defined by a non-principal ultrafilter, which allows one to take limits of bounded sequences. Once the ultrafilter is chosen, then the limit will commute with maps, and we get naturality of the averaging operator. 

Since $C$ is obtained from $I\times\partial_iC$ by attaching $1$-handles, there exists a retract $r\colon C\to\partial_iC$ such that $r\circ\iota=\mathrm{id}_{\partial_iC}$. Let $\beta_i\in\hat H^3(T_{\Phi_i})$ be any element with $[\beta_i,[T_{\Phi_i}]]=1$. Let $\beta=\hat A\circ\hat r^\ast\circ\hat{\rho_i}^\ast(\beta_i)$. Since $||\cdot||_\infty$ is non-increasing under pullbacks and under $\hat A$, we have $||\beta||_\infty\le||\beta_i||_\infty$. Since $\hat{\iota}^\ast\circ \hat{A}\circ \hat{r}^\ast\circ\hat{\rho_i}^\ast =\hat{A}\circ \hat{\iota}^\ast\circ \hat{r}^\ast\circ \hat{\rho_i}^\ast =\hat{A}\circ\hat{\rho_i}^\ast = \mathrm{id}$, we have $\hat\iota^\ast(\beta)=\beta_i$, and hence $[\beta,[T_\Phi]]=1$. It follows that $\inf\{||\beta||_\infty\mid\beta\in\hat H^3(T_\Phi),[\beta,[T_\Phi]]=1\}\le\inf\{||\beta_i||_\infty\mid\beta_i\in\hat H^3(T_{\Phi_i}),[\beta_i,[T_{\Phi_i}]]=1\}$, implying that $||T_{\Phi_i}||\le||[T_\Phi]||_1$ by Proposition~\ref{prop:l1_linfty}.
\end{proof}

Theorem~\ref{thm:volume} is now a consequence of Theorem~\ref{thm:CG} and Theorem~\ref{thm:Gromov_norm_compressionbody}.

\section{Monotonicity of dilatation constant}\label{sec:dilatation}
To prove monotonicity of dilatation, we exploit the relation between dilatation and the growth rate of group endomorphisms. The following definition and facts can be found in \cite[Expos\'e~10]{fathi1979travaux}.  

Let $G$ be a finitely generated group and $A\colon G\to G$ be an endomorphism of $G$. The \textit{growth rate} of $A$ is $$\gamma_A:=\sup_{g\in G}\limsup_{n\to\infty}\frac1n\log(\ell(A^n(g))),$$ where $\ell\colon G\to\Z_{\ge0}$ is the word length function on $G$ with respect to a finite generating set. The growth rate of $A$ is invariant under conjugations of either the domain or the target. In particular, if $f\colon X\to X$ is a continuous map between path-connected topological spaces, the induced map $f_*\colon\pi_1X\to\pi_1X$, well-defined up to conjugation, has a well-defined growth rate $\gamma_{f_*}$, independent of the choice of base points.

We specialize to the study of a surface homeomorphism $\phi\colon S\to S$. By Nielsen--Thurston theory, after an isotopy of $\phi$ (not necessarily rel boundary), we may pick a canonical reduction system $\delta\subset S$ of $\phi$ which is a multicurve preserved by $\phi$, and an open tubular neighborhood $\nu(\delta)$ preserved by $\phi$. Furthermore, we can write $S\backslash\nu(\delta)=\sqcup_{i=1}^kS_i$ so that $\phi$ restricts to each $S_i$, and each restriction $\phi|_{S_i}$ is either periodic or pseudo-Anosov, permuting the components of $S_i$ cyclically. Let $\lambda_i$ denote the dilatation constant of $\phi|_{S_i}$ if this is pseudo-Anosov, and let $\lambda_i=1$ otherwise. We are interested in the maximal dilatation $\lambda([\phi])=\lambda(\phi)=\max_i(\lambda_i)$. When $\phi$ is the monodromy of a fibered knot $K$, this is the dilatation $\lambda(K)$ of $K$ we defined in the introduction.

\begin{Lem}\label{lem:growth_rate}
If $S$ is connected, then $\log\lambda([\phi])=\gamma_{\phi_*}$.
\end{Lem}
\cite[Expos\'e~10]{fathi1979travaux} proved Lemma~\ref{lem:growth_rate} when $\phi$ is pseudo-Anosov, although the proof for the $\le$ direction works in general. We sketch a proof of the equality in its full generality, borrowing ideas from \cite{fathi1979travaux}.
\begin{proof}
We pick a metric on $S$ as follows: 
\begin{itemize}
\item On each $S_i$ where $\phi$ is periodic, we pick any $\phi$-invariant Riemannian metric;
\item On each $S_i$ where $\phi$ is pseudo-Anosov, we pick the singular Euclidean metric induced by a pair of stable and unstable measured foliations associated with $\phi|_{S_i}$;
\item On each component $A\cong I\times S^1$ of $\overline{\nu(\delta)}$, we fix any metric extending those on $\partial A$, so that the distance between any two points on the same boundary component of $A$ is achieved by an arc on the boundary, which is smaller than the distance between the two boundary components of $A$.
\end{itemize}
As in \cite[p. 190]{fathi1979travaux}, define $G_{[\phi]}:=\sup_x\limsup_{n\to\infty}\frac1n\log(\ell(\phi^n(x)))$, where $x$ runs over free homotopy classes of loops in $S$ and $\ell(x)$ denotes the minimal length of a loop representing $x$, with respect to the metric defined above. We prove that
\begin{equation}\label{eq:growth_rate}
\log\lambda([\phi])=G_{[\phi]}=\gamma_{\phi_*}.
\end{equation}
Since free homotopy classes of loops are exactly elements in $\pi_1S$ up to conjugacy, by exploiting the quasi-isometric embedding $\pi_1S\subset\tilde S$, we see that $G_{[\phi]}\le\gamma_{\phi_*}$.

Assume $\lambda([\phi])=\lambda_1$. If $\phi|_{S_1}$ is pseudo-Anosov, pick an essential loop $\alpha\subset S_1$ that is quasi-transverse to the stable foliation; then $\ell(\phi^n([\alpha]))\ge\lambda_1^nc$ for some $c>0$, implying that $\log\lambda([\phi])\le G_{[\phi]}$. If $\phi|_{S_1}$ is periodic, then $0=\log\lambda([\phi])\le G_{[\phi]}$ trivially holds.

Pick a base point $*\in S$ and a path $\gamma$ from $*$ to $\phi(*)$ to define $\phi_*\colon\pi_1(S,*)\to\pi_1(S,*)$. Let $g\in\pi_1(S,*)$ be represented by a based loop $\alpha_0$. Then $\phi_*^n(g)$ is represented by the based loop $\alpha_n$, inductively defined by $\alpha_n=\gamma*\phi(\alpha_{n-1})*\gamma^{-1}$. We estimate the length of $\alpha_n$, denoted $L(\alpha_n)$, by breaking into parts that lie in periodic pieces $S_i$, reducible pieces $A_j\subset\overline{\nu(\delta)}$, or pseudo-Anosov pieces $S_k$: $$L(\alpha_n)=L_P(\alpha_n)+L_R(\alpha_n)+L_A(\alpha_n),$$ where $L_P(\alpha_n)\le C$, $L_R(\alpha_n)\le Cn$, and $L_A(\alpha_n)\le\lambda([\phi])L_A(\alpha_{n-1})+C$, for some $C=C(\alpha_0)>0$. It follows that $\ell(\phi_*^n(g))\le L(\alpha_n)\le n\lambda([\phi])^nC$ for some $C>0$, which implies $\gamma_{\phi_*}\le\log\lambda([\phi])$.
\end{proof}

\begin{Thm}\label{thm:dilatation_compressionbody}
Let $C$ be a compression body with $\partial_iC,\partial_eC$ connected and nonempty, and let $\Phi\colon C\to C$ be a homeomorphism that restricts to $\Phi_i=\Phi|_{\partial_iC},\Phi_e=\Phi|_{\partial_eC}$. Then $\gamma_{\Phi_{i,*}}\le\gamma_{\Phi_*}\le\gamma_{\Phi_{e,*}}$.
\end{Thm}
\begin{proof}
We have maps $\pi_1(\partial_iC)\hookrightarrow\pi_1C\twoheadleftarrow\pi_1(\partial_eC)$ that intertwine the automorphisms $\Phi_{i,*}$, $\Phi_*$, and $\Phi_{e,*}$. Since the word length does not increase under surjections, we see $\gamma_{\Phi_*}\le\gamma_{\Phi_{e,*}}$. Since $\pi_1(\partial_iC)\hookrightarrow\pi_1C=\pi_1(\partial_iC)*F_k$ admits a retract, it is a quasi-isometric embedding. Hence the word length is coarsely preserved, and we see $\gamma_{\Phi_{i,*}}\le\gamma_{\Phi_*}$.
\end{proof}

Theorem~\ref{thm:dilatation} is a consequence of Theorem~\ref{thm:CG}, Lemma~\ref{lem:growth_rate} and Theorem~\ref{thm:dilatation_compressionbody}.\medskip

We now turn to Theorem~\ref{thm:finite}. The proof is by a careful analysis of the JSJ decomposition of knot complements, using Theorem~\ref{thm:dilatation} to bound the complexity of hyperbolic pieces in the decomposition. We refer to \cite{budney2006jsj} for a nice exposition of JSJ decomposition of knot complements.

\begin{proof}[Proof of Theorem~\ref{thm:finite}]
Pick any $J\le K$ that is not the unknot. Let $\Gamma_J$ denote the JSJ graph of $S^3\backslash J$, which is a rooted tree, where the root vertex $*$ corresponds to the JSJ piece adjacent to $J$. Every JSJ torus $T$, as well as the cusp boundary $\partial\nu(J)$, has a canonical parametrization $(m,\ell)$, well-defined up to overall sign, characterized by
\begin{itemize}
\item $m$ (the \textit{meridian}) is null-homologous in the complementary region of $T\subset S^3$ containing $J$;
\item $\ell$ (the \textit{longitude}) is null-homologous in the complementary region of $T\subset S^3$ not containing $J$;
\item $(m,\ell)$ is a positive basis when $T$ is oriented as the boundary of the solid torus it bounds in $S^3$.
\end{itemize}
We pick, once and for all, $(m,\ell)$ among the two choices, one for each JSJ torus $T$.

Conversely, the manifold $S^3\backslash J$, hence $J$ \cite{mca1989knots}, can be uniquely recovered from the following data in the natural way:
\begin{enumerate}
\item The JSJ graph $\Gamma_J$;
\item For every vertex $v$ of $\Gamma_J$, a compact oriented $3$-manifold $M_v$ with toroidal boundary;
\item For every vertex $v$ of $\Gamma_J$, a bijection $\pi_0(\partial M_v)\cong E(v)$ if $v\ne*$ and $\pi_0(\partial M_v\backslash\partial\nu(J))\cong E(v)$ if $v=*$, where $E(v)$ is the set of edges adjacent to $v$;
\item For every vertex $v$ of $\Gamma_J$ and every component $T$ of $\partial M_v$, a parametrization of $T$ by a pair of curves $(m,\ell)$.
\end{enumerate}
We show that the choices of each of (1)--(4) are finite.

Let $\phi_J\colon F_J\to F_J$ denote the monodromy of $J$, which is decomposed along the canonical reduction system into periodic and pseudo-Anosov pieces. Recall that $g(F_J)=g(J)\le g(K)$ \cite{zemke2019knot} and $\Delta_J|\Delta_K$ \cite{gilmer1984ribbon}.

(1) The canonical reduction system $\delta$ consists of a multicurve on $F_J$ without parallel components, hence its size is bounded. JSJ tori in $S^3\backslash J$ are in one-to-one correspondence with orbits of components of $\delta$ under $\phi_J$, hence have bounded size as well. It follows that the size of $\Gamma_J$ is bounded, hence the number of possible isomorphism types of the rooted tree $\Gamma_J$ is bounded.

(3) The number of choices is obviously finite.

(2)(4) We divide the proof into Seifert fibered and hyperbolic cases. Before we start, observe that $J$ has nonzero winding number in the solid torus bounded by any JSJ torus $T$. This is because otherwise we may assume $F_J\cap T=\emptyset$, implying that $\pi_1(F_J)=[\pi_1(S^3\backslash J),\pi_1(S^3\backslash J)]$ contains $\pi_1(T)=\Z^2$ as a subgroup, a contradiction.

Fix a vertex $v$ of $\Gamma_J$.

\textbf{Case 1}: $M_v$ is Seifert fibered.

By classification of Seifert fibered submanifolds of $S^3$, $M_v$ is homeomorphic to one of the following three families of manifolds \cite{budney2006jsj}:
\begin{itemize}
\item The complement of a keychain link $H_n\subset S^3$, $n\ge2$, defined as the unknot in $S^3$ together with $n$ of its distant meridians. This corresponds to a connected sum operation of knots. (In fact, in this case $v$ must be the root vertex.)
\item The complement of a $(p,q)$-cable pattern in $S^1\times D^2$ for some coprime $p,q$, where $p>1$. This corresponds to a $(p,q')$-cable operation of knots for some $q'\equiv q\pmod p$.
\item The complement of a $(p,q)$-torus knot for some coprime $p,q$, where $p,|q|>1$. This corresponds to a $T_{p,q}$ companion knot.
\end{itemize}
If $M_v$ is a keychain link complement, its complexity (namely the parameter $n$) is bounded by the size of $\Gamma_J$, which is bounded. The choice of allowable parametrizations $(m,\ell)$ on the boundaries of $M_v$ is unique up to symmetries of $M_v$ and signs.

If $M_v$ is a cable complement, its two boundary tori exhibit $J$ as an iterated satellite $J=P(C_{p,q}(C))$, where $P$ is the identity pattern if $v=*$. The boundary parametrization data of $M_v$ is determined by $(p,q)$ up to symmetries of $M_v$ and signs. Thus, we only need to bound the complexity of the pair $(p,q)$, where $p>1$ and $p,q$ are coprime.

Recall the following elementary facts.

\begin{Lem}\label{lem:satellite_genus_Alexander}
The genus and the Alexander polynomial of a satellite knot are given by $$g(P(C))=|w(P)|g(C)+\hat g(P),\ \Delta_{P(C)}(t)=\Delta_{P(U)}(t)\Delta_C(t^{w(P)}),$$ where $w(P)$ is the winding number of $P$, and $\hat g(P)$ is the minimal genus of an orientable surface in the solid torus cobounding $P$ and $w(P)$ copies of canonical longitudes on the boundary.\qed
\end{Lem}

By Lemma~\ref{lem:satellite_genus_Alexander}, we have $$\Delta_J(t)=\Delta_{P(U)}(t)\Delta_{T_{p,q}}(t^{w(P)})\Delta_C(t^{pw(P)}).$$ Since $w(P)\ne0$, $\Delta_{T_{p,q}}(t)=(t^{pq}-1)(t-1)/(t^p-1)(t^q-1)$, and $\Delta_J(t)|\Delta_K(t)$ is bounded, this bounds the complexity of $(p,q)$ unless $q=\pm1$. On the other hand, Lemma~\ref{lem:satellite_genus_Alexander} implies that $p\le p\cdot g(C)\le g(J)\le g(K)$ is bounded, hence the complexity of $(p,q)$ is bounded in any case.

If $M_v$ is a torus knot complement, then its boundary torus exhibits $J$ as a satellite $J=P(T_{p,q})$ where $P$ is the identity satellite pattern if $v=*$. The manifold $M_v$ together with its boundary parametrization data up to sign is determined from the pair $(p,q)$, where $p,|q|>1$ are coprime. By the Alexander polynomial bound as above, the complexity of $(p,q)$ is bounded.

\textbf{Case 2}: $M_v$ is hyperbolic.

(2) Hyperbolic pieces in the JSJ decomposition of $S^3\backslash J$ are in one-to-one correspondence with mapping tori of pseudo-Anosov pieces in the Nielsen--Thurston canonical decomposition of $\phi_J$. Since the fiber surface of each such mapping torus has bounded complexity $-\chi\le-\chi(F_J)\le2g(K)-1$, and since the dilatation constant of the monodromy is bounded by $\lambda(J)\le\lambda(K)$ by Theorem~\ref{thm:dilatation}, there are only finitely many possible hyperbolic mapping tori that may arise.

(4) First, we bound the number of longitudes $\ell$ on boundary components of $M_v$. Put the fiber surface $F_J$ in minimal position with respect to $\partial M_v$. Then the intersection between $F_J$ and a boundary $T\subset\partial M_v$ is $w$ parallel longitudes of $T$, where $w\ne0$ is the winding number of $J$ inside the solid torus bounded by $T$, determined up to sign. Moreover, $F_J\cap M_v$ is Thurston-norm-minimizing, with complexity bounded above by $-\chi(F_J)\le2g(K)-1$. Since $M_v$ is hyperbolic, the Thurston norm on $H_2(M_v,\partial M_v)$ is non-degenerate, implying that the possible number of homology classes of $[F_J\cap M_v]$ is finite. In particular, the number of possible longitudes on components $\partial M_v$, which are determined by $\partial[F_J\cap M_v]\in H_1(\partial M_v)$, is also finite.

Next, we bound the number of meridians $m$ on boundary components of $M_v$. Let $T_*,T_1,\cdots,T_n$ be the components of $\partial M_v$, where $T_*$ is the component closest to $J$ in $S^3$. By the Fox re-embedding theorem (see \cite[Proposition~2]{budney2006jsj}), the Dehn filling of $T_*\subset M_v$ by slope $m$ is homeomorphic to $\natural^n(S^1\times D^2)$. By bounds on exceptional surgeries, or more precisely the knot complement problem \cite{mca1989knots} when $n=0$ and bounds on boundary-reducible surgeries \cite{wu1992incompressibility} when $n>0$, there are finitely many such $m$ on $T_*$. If $n>0$, once the meridian $m$ on $T_*$ and all longitudes $\ell$ on, say $T_2,\cdots,T_n$, are chosen, the Dehn filling of $M_v$ by these slopes has a unique boundary component $T_1$, and the meridian $m$ on $T_1$ is uniquely characterized (up to sign) by being null-homologous in this filling. This proves the finiteness of the choice of meridians on $\partial M_v$.
\end{proof}

\section{Algorithmic compressions of surface homeomorphisms}\label{sec:compression}
In this section, we prove Theorem~\ref{thm:compression}, generalizing the main results of Casson--Long \cite{casson1985algorithmic} (Theorem~\ref{thm:CL}). The extension is obtained by a careful case-by-case analysis when the Nielsen--Thurston decomposition of the surface homeomorphism $\phi\colon S\to S$ is nontrivial. We remark that \cite{casson1985algorithmic} only considered the case when $S$ is closed and connected, but relevant techniques for pseudo-Anosov homeomorphisms extend to the case with boundary without much change, and the connectedness assumption can also be dropped for our convenience.

\begin{Rmk}
By Lemma~\ref{lem:centralizer}, a generating set of the symmetry group of the pair $(S,\phi)$ appearing in Theorem~\ref{thm:compression} is also computable.
\end{Rmk}

\subsection{Preliminaries and examples}\label{sbsec:comp_ex}
A compression $\Phi\colon C\to C$ of $\phi$ is \textit{nontrivial} if $C$ is not the product cobordism $I\times S$; it is \textit{minimal} if $C$ does not decompose into a composition of two nontrivial compressions such that $\Phi$ restricts to each of them.

\begin{Ex}\label{ex:no_symmetry}
The pseudo-Anosov assumption is necessary in Theorem~\ref{thm:CL} as stated. For instance, if $S$ has a component that is not a sphere, disk, annulus, or pair of pants, and $\phi$ is the identity map, then $\phi$ admits infinitely many minimal compressions, one for each isotopy class of essential simple closed curve in $S$.
\end{Ex}

In view of Example~\ref{ex:no_symmetry}, it is necessary to account for the symmetries in order to obtain a finiteness result when the pseudo-Anosov assumption is dropped. 

If $f\colon S\to S$ is a surface homeomorphism, we define the \textit{pushforward} of a compression $C$ of $\phi$ by $f$, denoted $f_*C$, to be the compression of $f_*\phi=f\phi f^{-1}\colon S\to S$ given by $C$ but with the exterior boundary reparametrized by $f$. Alternatively, if $S$ is connected, the compression $C$ of $\phi$ is determined by the $\phi_*$-invariant normal subgroup $N=\ker(\pi_1S\to\pi_1C)$, and the pushforward of $C$ by $f$ is determined by the $(f_*\phi)_*$-invariant normal subgroup $f_*N$.

If $c\colon S\to S$ is a surface homeomorphism rel boundary that commutes with $\phi\colon S\to S$ up to isotopy rel boundary, then $c_*\phi$ is isotopic to $\phi$ rel boundary; hence, for a compression $C$ of $\phi$, the pushforward $c_*C$ can be naturally regarded as a compression of $\phi$ again. The \textit{symmetry group} of $(S,\phi)$, denoted $C(S,\phi)$, is the subgroup of the rel boundary mapping class group of $S$ represented by such surface homeomorphisms $c$. Then, $C(S,\phi)$ acts on the set of compressions of $\phi$. The action of $C(S,\phi)$ is the symmetry that we mod out in the statement of Theorem~\ref{thm:compression}.

Since the centralizer of a pseudo-Anosov mapping class $[\phi]$ (with free boundary condition) is a finite extension of the infinite cyclic group generated by $[\phi]$ \cite{mccarthy1982normalizers}, and $\phi_*$ acts trivially on the set of compressions of $\phi$, Theorem~\ref{thm:compression}(1), in the case when $\phi$ is pseudo-Anosov, is equivalent to Theorem~\ref{thm:CL}(1). In fact, although $\phi$ itself may not be a symmetry of $(S,\phi)$ since it may not be rel boundary, some power $\phi^N$ that permutes $\pi_0(\partial S)$ trivially would be freely isotopic to a symmetry of $(S,\phi)$.

\begin{Ex}
Let $C$ be the compression body with $\partial_iC=T^2$ and $\partial_eC=\Sigma_3$, where $\Sigma_3$ is the genus $3$ closed surface. The proof of \cite[Lemma~3.4]{biringer2013extending} shows that there exist two simple closed curves $\alpha,\beta$ on $\Sigma_3$ that compress in $C$ so that $\alpha\cup\beta$ (in minimal position) fills $\Sigma_3$. By Thurston \cite[Theorem~7]{thurston1988geometry}, many compositions of Dehn twists about $\alpha,\beta$ (e.g. $T_\alpha\circ T_\beta^{-1}$) are pseudo-Anosov. Since $T_\alpha,T_\beta$ each compress in $C$ to $\mathrm{id}_{T^2}\colon T^2\to T^2$, we know any such pseudo-Anosov homeomorphism $\phi\colon\Sigma_3\to\Sigma_3$ compresses in $C$ to $\mathrm{id}_{T^2}$. Since $T^2$ has infinitely many isotopy classes of simple closed curves, we obtain infinitely many handlebody compressions of $\phi$. This example shows that the minimality assumption in Theorem~\ref{thm:CL} is necessary, hence the minimality assumption in Theorem~\ref{thm:compression} is also necessary in view of the previous paragraph.
\end{Ex}

\begin{Ex}
Let $S$ be a surface with nonempty boundary and $\phi\colon S\to S$ be a pseudo-Anosov homeomorphism. Let $D(S)$ denote the double of $S$ along the boundary, and $D(\phi)\colon D(S)\to D(S)$ the double of $\phi$. Then $D(\phi)$ is a reducible surface homeomorphism with two Nielsen--Thurston pieces, which compresses in the handlebody $I\times S$. In fact, it will follow from the argument in the next section that $I\times S$ is a minimal compression of $D(\phi)$; in particular, the compression is not supported on a single Nielsen--Thurston piece. This example shows that Theorem~\ref{thm:compression} does not reduce directly to the periodic and pseudo-Anosov cases.
\end{Ex}

\subsection{Finiteness of minimal compressions}\label{sbsec:comp_finite}
Throughout the rest of Section~\ref{sec:compression}, we may assume without loss of generality that $\phi\colon S\to S$ acts transitively on the set of connected components of $S$. Then, if $S\ne\emptyset$, compressions of $\phi$ are in natural one-to-one correspondence with compressions of the restriction of $\phi^{\#\pi_0(S)}$ to one connected component of $S$; thus, we may further assume $S$ is connected. The cases when $S$ is empty, a disk, a sphere, or an annulus are easy, so we assume $S\ne\emptyset$ is hyperbolic.

Let $\delta\subset S$ denote the canonical reduction system of $\phi\colon S\to S$. Write $S\backslash\nu(\delta)=\sqcup_{i=1}^\ell S_i$ so that $\phi$ (after an isotopy) restricts to either a periodic or a pseudo-Anosov homeomorphism on each $S_i$, permuting its components cyclically. By assumption, each component of each $S_i$ is hyperbolic.

\begin{proof}[Proof of Theorem~\ref{thm:compression}(1)]
Let $\Phi\colon C\to C$ be a minimal compression of $\phi$.

Pick an essential (simple closed) curve $\gamma\subset S$ that compresses in $C$ and is in general position with respect to $\delta$, with $|\gamma\cap\delta|$ minimized. By a careful cut-and-paste argument applied to $\gamma$ and a (pulled-tight) iterate of $\gamma$ under a power of $\phi$, we will show that $|\gamma\cap\delta|\in\{0,2,4\}$, and later show that $|\gamma\cap\delta|=4$ cannot occur. Then we treat $|\gamma\cap\delta|=0$ (three subcases depending on whether $\gamma$ is parallel to a canonical reduction curve, essentially contained in a periodic piece, or essentially contained in a pseudo-Anosov piece) and $|\gamma\cap\delta|=2$ (three subcases, giving rise to product or twisted product regions in the compression body $C$). In each of the six surviving subcases, the compression is determined by finite data up to symmetries, giving the desired finiteness. Later on, Section~\ref{sbsec:comp_forms} will build on the proof to give descriptions of the six different forms of $C$.

Pick hyperbolic metrics on the complementary pieces $S_i$ with totally geodesic boundaries, and Euclidean metrics on $\nu(\delta)$ so that $\delta$ is geodesic. We may pull $\gamma$ tight, so that it is geodesic on each $S_i$ perpendicular to the boundaries, and geodesic on $\nu(\delta)$ (if there are parallel arcs on $\gamma$ that get identified to a single arc in this procedure, perturb them slightly apart to make sure $\gamma$ is still a simple closed curve).

Pick $N>0$ so that $\phi^N$ preserves each component of $\partial\nu(\delta)$ (hence also of $S\backslash\nu(\delta)$), and that $\phi^N$ restricts to the identity map on the periodic pieces of $\phi$. Pick a large $q>0$. Pull $\phi^{Nq}(\gamma)$ tight to some simple closed curve $\gamma_q$ in the same procedure as before. Then perturb $\gamma_q$ slightly on a neighborhood of the periodic pieces, so that $\gamma,\gamma_q,\delta$ are in pairwise minimal position.

If $\gamma\cap\delta$ is nonempty, then either $\gamma$ enters some pseudo-Anosov piece of $S$, or otherwise $\gamma$ and $\phi^{Nq}(\gamma)$ differ by some nontrivial multi-twist along a multicurve on $S$ consisting of copies of components of $\delta$ that intersect $\gamma$. In the first case, $\gamma\cap\gamma_q$ is nonempty for large $q$. In the second case, $\gamma\cap\gamma_q$ is nonempty for $q\ge2$ since in such case some intersection between $\gamma$ and $\phi^{Nq}(\gamma)$ on $\nu(\delta)$ cannot be perturbed away.

By assumption, $\gamma,\gamma_q$ bound some disks $D,D'$ in $C$, respectively. Since $C$ is irreducible, isotoping $D'$ rel boundary if necessary, we may assume $D$ and $D'$ intersect transversely in some disjoint union of double arcs. If $e$ is a double arc that is innermost on $D$, in the sense that there is some arc $\alpha\subset\gamma$ with $\partial\alpha=\partial e$ and $int(\alpha)\cap\gamma_q=\emptyset$, then we may pick an arc $\beta\subset\gamma_q$ with $\partial\beta=\partial e$ and $|\beta\cap\delta|\le\tfrac12|\gamma_q\cap\delta|=\tfrac12|\gamma\cap\delta|$, and apply a cut-and-paste along $e$ to obtain an essential curve $\gamma':=\alpha\cup\beta\subset S$ that compresses in $C$. Some choice of $\alpha$ would contradict the minimality of $|\gamma\cap\delta|$ unless every such innermost $\alpha$ on $\partial D$ (and innermost $\beta$ on $\partial D'$) intersects $\delta$ exactly $\tfrac12|\gamma\cap\delta|$ times. Moreover, if $\gamma\cap\gamma_q$ contains a point near $\delta$ (i.e. inside $\nu(\delta)$, or equivalently, $\gamma,\gamma_q,\delta$ cobound a triangle), then we may move such a point to the other side of $\delta$ to arrange the existence of some innermost $\alpha$ that contradicts the minimality of $|\gamma\cap\delta|$. It follows, in the case $\gamma\cap\delta\ne\emptyset$, that
\begin{itemize}
\item $D\cap D'$ consists of parallel arcs on $D$ (resp. $D'$);
\item $\gamma\cap\delta=A\sqcup B$ for some sets $A,B$ with the same cardinality, such that each double arc divides $D$ into two parts containing $A$ and $B$, respectively; similarly for $\gamma_q$;
\item No intersection point in $\gamma\cap\gamma_q$ is near $\delta$.
\end{itemize}
Consequently, when $\gamma\cap\delta\ne\emptyset$, all intersections $\gamma\cap\gamma_q$ lie essentially in some connected components $\Sigma_1,\Sigma_2$ (not necessarily distinct) of $S\backslash\nu(\delta)$, with those in $\Sigma_1$ being in bijection with those in $\Sigma_2$, the bijection given by arcs in $D\cap D'$. Moreover, both $\Sigma_1$ and $\Sigma_2$ are contained in pseudo-Anosov pieces of $S$, since otherwise the intersection points would be near $\delta$. See Figure~\ref{fig:D} for an illustration when $|\gamma\cap\delta|=4$. The intersection $\gamma\cap\partial(\nu(\delta))$ divides $\gamma$ into consecutive segments lying in $\Sigma_1,\nu(\delta_1),\Sigma_1',\nu(\delta_1')\cdots,$ $\Sigma_r',\nu(\delta_r'),\Sigma_2,\nu(\delta_2),\Sigma_1'',\nu(\delta_1''),\cdots,\Sigma_r'',\nu(\delta_r'')$ for some connected components $\Sigma_i',\Sigma_i''$ of $S\backslash\nu(\delta)$ and $\delta_1,\delta_2,\delta_i',\delta_i''$ of $\delta$, $i=1,\cdots,r$, $r\ge0$. The same description applies to $\gamma_q$. As there are no intersections between $\gamma\cap\gamma_q$ in $\Sigma_i'$ and $\Sigma_i''$, taking $q$ large implies that all $\Sigma_i'$ and $\Sigma_i''$ are contained in periodic pieces of $S$. Also, by taking $q\ge2$, we see that $r\le1$, since otherwise $\gamma$ and $\gamma_q$ would intersect near $\delta_1'$.

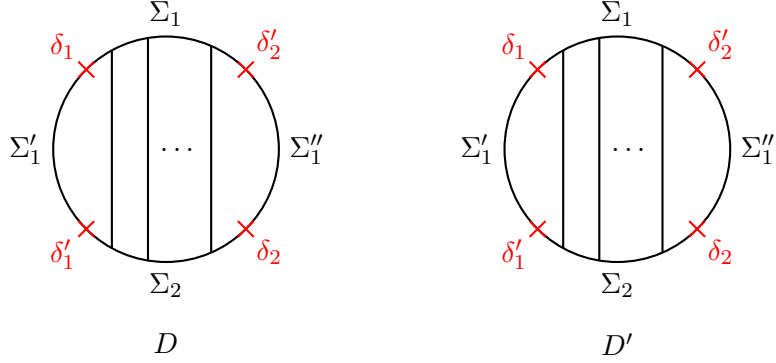
\begin{figure}
\begin{tikzpicture}[
    xscale=.6,
    yscale=.6,
    thick,
    >=stealth,
    red cross/.style={
        cross out,
        draw=red,
        minimum size=6pt, 
        inner sep=0pt,
        outer sep=0pt,
        thick
    }
]

\def\R{2.5}

\begin{scope}[local bounding box=leftfig]
    \draw (0,0) circle (\R);

    \node[above] at (90:\R) {$\Sigma_1$};
    \node[below] at (270:\R) {$\Sigma_2$};
    \node[left]  at (180:\R) {$\Sigma_1'$};
    \node[right] at (0:\R)   {$\Sigma_1''$};

    \foreach \x in {-1.2, -0.4, 1.0} {
        \pgfmathsetmacro{\y}{sqrt(\R*\R - \x*\x)}
        \draw (\x, \y) -- (\x, -\y);
    }
    \node at (0.3,0) {$\dots$};

    \node[red cross] (d1) at (135:\R) {};
    \node[red, above left] at (d1) {$\delta_1$};

    \node[red cross] (d1p) at (225:\R) {};
    \node[red, below left] at (d1p) {$\delta_1'$};

    \node[red cross] (d2p) at (45:\R) {};
    \node[red, above right] at (d2p) {$\delta_2'$};

    \node[red cross] (d2) at (315:\R) {};
    \node[red, below right] at (d2) {$\delta_2$};

    \node[below=0.8cm] at (0,-\R) {$D$};
\end{scope}

\begin{scope}[xshift=10cm, local bounding box=rightfig]
    \draw (0,0) circle (\R);

    \node[above] at (90:\R) {$\Sigma_1$};
    \node[below] at (270:\R) {$\Sigma_2$};

    \foreach \x in {-1.2, -0.4, 1.0} {
        \pgfmathsetmacro{\y}{sqrt(\R*\R - \x*\x)}
        \draw (\x, \y) -- (\x, -\y);
    }
    \node at (0.3,0) {$\dots$};

    
    \node[red cross] (rd1) at (135:\R) {};
    \node[red, above left] at (rd1) {$\delta_1$};
    
    \node[red cross] (rd1p) at (225:\R) {};
    \node[red, below left] at (rd1p) {$\delta_1'$};
    
    \node[left] at (180:\R) {$\Sigma_1'$};

    \node[red cross] (rd2p) at (45:\R) {};
    \node[red, above right] at (rd2p) {$\delta_2'$};
    
    \node[red cross] (rd2) at (315:\R) {};
    \node[red, below right] at (rd2) {$\delta_2$};

    \node[right] at (0:\R) {$\Sigma_1''$};

    \node[below=0.8cm] at (0,-\R) {$D'$};
\end{scope}

\end{tikzpicture}
\caption{The disks $D,D'$ with their double arcs and boundary data in the case $|\gamma\cap\delta|=4$. The components $\Sigma_1,\Sigma_2$ are pseudo-Anosov, and the components $\Sigma_1',\Sigma_1''$ are periodic.}
\label{fig:D}
\end{figure}

We have thus shown, in particular, that $|\gamma\cap\delta|\in\{0,2,4\}$. In fact, we will further rule out the case $|\gamma\cap\delta|=4$ later in the proof.

\textbf{Case 1}: $|\gamma\cap\delta|=0$.

\textbf{Case 1.1}: $\gamma$ is parallel to some component $\delta_1$ of $\delta$ or some component $\partial_1S$ of $\partial S$.

Then, $C$ is the compression body that compresses each component of the multicurve $\cup_{i\ge0}\phi^i(\delta_1)\subset S$ or each component of (a slight push-in of) the multicurve $\cup_{i\ge0}\phi^i(\partial_1S)\subset S$.

\textbf{Case 1.2}: $\gamma$ is essentially contained in some periodic piece $S_1$ of $S$.

Pick a hyperbolic metric on $S_1$ invariant under the periodic map $\phi_1:=\phi|_{S_1}$, with respect to which $\gamma$ is geodesic. We may assume $\gamma$ is chosen to have minimal length. Then, by a cut-and-paste argument as before (cf. \cite[Lemma~4.2]{bonahon1983cobordism}), each $\phi_1^i(\gamma)$ is either disjoint from $\gamma$ or identical to $\gamma$. Hence, $C$ is the compression body that compresses each component of the multicurve $\cup_{i\ge0}\phi^i(\gamma)$.

It remains to prove that the number of such $\gamma$ on $S_1$ is finite up to symmetries of $(S_1,\phi_1)$, namely up to homeomorphisms of $S_1$ rel boundary that commute with $\phi_1$ up to isotopy rel boundary (note that such a homeomorphism extends by the identity map to a symmetry of $(S,\phi)$). In fact, the homeomorphisms we construct will commute strictly with $\phi_1$.

Let $K$ denote the period of $\phi_1$. Then $\phi_1$ is the positive generator of the deck transformation group for the $K$-fold cyclic cover $\pi\colon S_1\to\overline{S_1}$, where $\overline{S_1}$ denotes the orbifold $S_1/\phi_1$. The curve $\gamma$ maps to an essential simple closed curve $\bar\gamma=\pi(\gamma)\subset\overline{S_1}$. It remains to prove that there exist only finitely many simple closed curves on $\overline{S_1}$ up to homeomorphisms that lift to rel boundary homeomorphisms of $S_1$.

Consider the rel boundary homeomorphism group $G_0=\mathrm{Homeo}_{\partial}^+(\overline{S_1})$ of $\overline{S_1}$, its subgroup $G_1$ of homeomorphisms that lift to $S_1$, and its subgroup $G_2$ of homeomorphisms that lift to rel boundary homeomorphisms of $S_1$. Since there are only finitely many $K$-fold coverings of $\overline{S_1}$, we know $G_1$ is of finite index in $G_0$. Since there are only finitely many automorphisms of $\partial S_1$ that intertwine with $\pi|_{\partial S_1}$, we know $G_2$ is of finite index in $G_1$. Now, since the topological type of the complement of a simple closed curve in $\overline{S_1}$ is restricted, there are only finitely many simple closed curves in $\overline{S_1}$ up to $G_0$. Consequently, there are only finitely many simple closed curves in $\overline{S_1}$ up to $G_2$, as desired.

\textbf{Case 1.3}: $\gamma$ is essentially contained in some pseudo-Anosov piece $S_1$ of $S$.

Fix a complete hyperbolic metric on $int(S_1)$, and let $\gamma$ be geodesic. The proof of \cite[Theorem~1.2]{casson1985algorithmic} extends to the punctured case and shows that the length of $\gamma$ is bounded above by a constant depending on $\phi|_{S_1}$ and the metric. This implies the finiteness of such $\gamma$, hence of minimal compressions $C$ containing such $\gamma$ by \cite[Corollary~2.5]{casson1985algorithmic}.

\textbf{Case 2}: $|\gamma\cap\delta|=2$.

Applying a further cut-and-paste operation if necessary, we may assume, for ease of notation, that $\delta_1=\delta_2$.

\textbf{Case 2.1}: $\Sigma_1\ne\Sigma_2$.

One can observe that in this case, the singular surface $D\cup D'$ is homeomorphic to $I\times((\gamma\cup\gamma_q)\cap\Sigma_1)$.

We may assume $q$ is large enough so that $\gamma$ and $\gamma_q$ \textit{fill} $\Sigma_i$, $i=1,2$, meaning that each complementary region of their union in $\Sigma_i$ is either a disk whose intersection with $\partial\Sigma_i$ has at most one component, or an annular neighborhood of a boundary component. One can deduce this fact (except in sporadic cases) from \cite[Proposition~4.6]{masur1999geometry} which says that a pseudo-Anosov homeomorphism (on a surface that is not a four-punctured sphere or a one-punctured torus) acts on the curve complex with positive translation length, and \cite[Lemma~4.10]{masur2013geometry} which says that the curve complex embeds quasi-isometrically into the arc and curve complex.

The union of closed faces of $D\cup D'$ adjacent to an interior disk component $D_0$ of the complement of $\gamma\cup\gamma_q$ in $\Sigma_1$ (resp. $\Sigma_2$) is an annulus with boundary $\partial D_0\sqcup\gamma'$ where $\gamma'$ is a simple closed curve on $\Sigma_2$ (resp. $\Sigma_1$). Since $\gamma'$ compresses in $C$, it is necessarily trivial on $S$, hence on $\Sigma_2$ (resp. $\Sigma_1$). Consequently, the interior boundary of $\nu(S\cup D\cup D')$ contains some spheres that each contain an interior disk component of the complement of $\gamma\cup\gamma_q$ in $\Sigma_1$ and one such disk component in $\Sigma_2$. Filling in these spheres by balls, we see a product region $I\times\Sigma_1\hookrightarrow C$ cobounding $\Sigma_1$ and $\Sigma_2$. There are four other sphere components of the interior boundary of $\nu(S\cup D\cup D')$ that each contain a boundary disk component of the complement of $\gamma\cup\gamma_q$ in $\Sigma_1$ and one such disk component in $\Sigma_2$. Filling in these spheres by balls as well, we obtain a product region $I\times\Sigma_1$ in $C$ as before, except that now the $I\times\partial_1\Sigma_1$ part of its boundary is glued to $\nu(\delta_1)$, where $\partial_1\Sigma_1$ is the boundary component of $\Sigma_1$ adjacent to $\nu(\delta_1)$. This product region gives rise to a compression body $C_{00}\subset C$ with exterior boundary $S$.

Observe that $\phi^i(\Sigma_1)=\Sigma_1$ or $\Sigma_2$ if and only if $\phi^i(\Sigma_2)=\Sigma_2$ or $\Sigma_1$, since otherwise a cut-and-paste operation applied to disks bounded by $\gamma$ and $\phi^{i+Nq}(\gamma)$ for large $q$ decreases $|\gamma\cap\delta|$. Therefore, the union $C'':=\cup_{i\ge0}\Phi^i(C_{00})\subset C$ is a compression body with exterior boundary $S$ that consists of disjoint product regions $P_i$ cobounding $\phi^i(\Sigma_1)$ and $\phi^i(\Sigma_2)$, $i=0,1,\cdots,M-1$, where $M$ is the minimal integer with $\phi^M(\Sigma_1)=\Sigma_1$ or $\Sigma_2$. A vertical annulus boundary of a product region $P_i$ adjacent to some $\phi^j(\nu(\delta_1))$ is glued to such $\phi^j(\nu(\delta_1))$. Since $\Phi$ restricts to $C''$ and $C$ is minimal, we see $C''=C$.

The compression body $C$ is determined by the product region $P_0$, which is in turn determined by an orientation-reversing surface homeomorphism $\theta\colon\Sigma_1^+\to\Sigma_2^+$ where $\Sigma_i^+$ is the union of $\Sigma_i$ and boundary annuli in $S\backslash int(\Sigma_i)$ bounded by $\cup_{j\ge0}\phi^j(\delta_1)$, $i=1,2$. A boundary of $\Sigma_i^+$ is called a \textit{fixed boundary} if it is a copy of some $\phi^j(\delta_1)$, and a \textit{free boundary} otherwise, $i=1,2$. The surface homeomorphism $\theta$ is rel the fixed boundaries of $\Sigma_1^+,\Sigma_2^+$ (which are canonically identified), and is considered up to isotopies rel the fixed boundaries. It remains to show that for any fixed $\Sigma_1,\Sigma_2,\delta_1$, the number of such $\theta\colon\Sigma_1^+\to\Sigma_2^+$ is finite up to symmetries.

\textbf{Case 2.1.1}: $\phi^M(\Sigma_1)=\Sigma_1$.

We first prove the following lemma.

\begin{Lem}\label{lem:commutator_no_boundary_twist}
Let $\Sigma$ be a compact oriented surface and $\alpha,\beta\colon\Sigma\to\Sigma$ be orientation-preserving surface homeomorphisms. Assume moreover that $\alpha$ is rel boundary, and $\beta$ permutes the boundary components trivially. If $[\alpha,\beta]$ is isotopic to $\mathrm{id}$, then it is isotopic to $\mathrm{id}$ rel boundary.
\end{Lem}
\begin{proof}
Without loss of generality, assume $\Sigma$ is nonempty and connected. By adding collars of isotopies to the boundaries, we may assume $\beta$ is also rel boundary. By assumption, $[\alpha,\beta]$ is isotopic rel boundary to some multi-twist $T_1^{n_1}\cdots T_k^{n_k}$, where $T_i$ is the Dehn twist along the $i$-th boundary component of $\Sigma$. Suppose, for the sake of contradiction, that $T_1^{n_1}\cdots T_k^{n_k}$ is not isotopic to $\mathrm{id}$ rel boundary. Then, by capping off some boundary components of $\Sigma$ if necessary, we may assume that $k=1$, $n_1\ne0$, and $\Sigma$ is not a disk.

\textit{Case 1}: $\Sigma$ has genus $0$.

By capping off all but the first two boundary components, we may assume $\Sigma=I\times S^1$. We obtain a contradiction since the rel boundary mapping class group of $\Sigma$ is $\Z$ generated by the Dehn twist.

\textit{Case 2}: $\Sigma$ has genus $1$.
 
By capping off all but the first boundary component, we may assume $\Sigma$ has only one boundary. Then, the rel boundary mapping class group of $\Sigma$ is isomorphic to the $3$-strand braid group $B_3$ by the Birman--Hilden theory; see \cite[Theorem~9.2]{farb2011primer}. Moreover, under the isomorphism, the boundary Dehn twist on $\Sigma$ corresponds to the square of the full twist in $B_3$, which maps to $12$ under the abelianization map $B_3\to(B_3)_{ab}\cong\Z$ given by summing the exponents in a word of Artin generators. In particular, any power of the boundary Dehn twist is not contained in the commutator subgroup of the mapping class group, a contradiction.

\textit{Case 3}: $\Sigma$ has genus at least $2$.

By \cite[Theorem~2]{baykur2013sections} (or \cite{baykur2014flat} for a proof without gauge theory) on the commutator length, $T_1^{n_1}$ is not equal to a commutator in the (rel boundary) mapping class group, a contradiction.
\end{proof}

We return to the proof of Case 2.1.1. Let $\theta,\theta'\colon\Sigma_1^+\to\Sigma_2^+$ be two orientation-reversing homeomorphisms rel the fixed boundaries, giving rise to product regions $P_0,P_0'$ in some minimal compressions $C,C'$, as described above. We claim that if $\theta,\theta'$ induce equal maps $\pi_0(\partial\Sigma_1^+)\to\pi_0(\partial\Sigma_2^+)$ (there are finitely many such maps), then $C,C'$ are related by a symmetry of $(S,\phi)$.

To this end, isotope $\theta'$ rel the fixed boundaries so that $\theta|_{\partial\Sigma_1^+}=\theta'|_{\partial\Sigma_1^+}$. Let $k$ be the minimal positive integer so that $\phi^{kM}|_{\Sigma_1}$ permutes the boundary components of $\Sigma_1$ trivially.

Since $\phi^M$ restricts to $\Sigma_1$ and $\Sigma_2$, $\Phi^M$ restricts to a map on the product region $P_0$ preserving $\Sigma_i^+$, $i=1,2$. Since the top and the bottom of $P_0$ are identified via $\theta$, we know that $\phi^M|_{\Sigma_1^+}$ is homotopic to, hence isotopic to $\theta^{-1}\circ\phi^M|_{\Sigma_2^+}\circ\theta$, rel the fixed boundaries. The same holds for $\theta'$. It follows that the rel boundary homeomorphism $\alpha:=\theta^{-1}\circ\theta'\colon\Sigma_1^+\to\Sigma_1^+$ commutes with $\beta_0:=\phi^M|_{\Sigma_1^+}$ up to isotopy rel the fixed boundaries. Therefore, $[\alpha,\beta_0]$ is isotopic rel boundary to some multi-twist $\prod_{i=1}^r\prod_{j=1}^{k_i}T_{i,j}^{n_{i,j}}$ where $i$ runs over orbits under $\beta_0$ of free boundary components of $\Sigma_1^+$, $j$ runs over free boundary components in the $i$-th orbit in (cyclic) order, and $T_{i,j}$ is the Dehn twist along the $(i,j)$-th free boundary component of $\Sigma_1^+$. Let $\beta:=\beta_0^k$, which is by assumption a surface homeomorphism of $\Sigma_1^+$ that permutes the boundary components trivially. Since $\beta_0T_{i,j}\beta_0^{-1}$ is isotopic rel boundary to $T_{i,j+1}$, we see that $[\alpha,\beta]$ is isotopic to $\prod_{i=1}^r\prod_{j=1}^{k_i}T_{i,j}^{N_i}$ rel boundary, for $N_i=\frac{k}{k_i}\sum_{j=1}^{k_i}n_{i,j}$, $i=1,\cdots,r$. On the other hand, $[\alpha,\beta]$ is isotopic to the identity rel boundary by Lemma~\ref{lem:commutator_no_boundary_twist}. Since $\Sigma_1^+$ is hyperbolic, the boundary Dehn twists are linearly independent, implying that $N_i=0$ for all $i$. Now, by isotoping $\theta'$ rel the fixed boundaries, we may replace $\theta'$ by $\theta'':=\theta'\prod_{i=1}^r\prod_{j=1}^{k_i}T_{i,j}^{r_{i,j}}$ where $r_{i,j}=\sum_{s=j+1}^{k_i}n_{i,s}$. Then $\alpha':=\theta^{-1}\circ\theta''$ commutes with $\beta_0=\phi^M|_{\Sigma_1^+}$ up to isotopy rel boundary. Now, define a surface homeomorphism $c\colon S\to S$ rel boundary by $\phi^i\circ\alpha'\circ\phi^{-i}$ on $\phi^i(\Sigma_1^+)$, $i=0,1,\cdots,M-1$, extended by the identity to the rest of $S$. Then $c$ commutes with $\phi$ up to isotopy rel boundary, thus defines a symmetry of $(S,\phi)$. One may easily check that $c_*C'=C$, proving the claim.

\textbf{Case 2.1.2}: $\phi^M(\Sigma_1)=\Sigma_2$.

We first prove the following lemma, the algorithmic part of which is not used until Section~\ref{sbsec:comp_algo}.

\begin{Lem}\label{lem:square_root}
Let $S$ be a connected oriented hyperbolic surface without boundary, with (possibly empty) punctures, and $f$ be a pseudo-Anosov mapping class on $S$. There is an algorithm to output all mapping classes $a$ on $S$, orientation-preserving or not, with $a^2=f$. In particular, there are at most finitely many such mapping classes $a$.
\end{Lem}
\begin{proof}
We compute the periodic splitting sequence of the stable lamination of $f$ in the sense of \cite{agol2011ideal}, which is a bi-infinite sequence of measured train tracks $\cdots\rightharpoonup(\tau_0,\mu_0)\rightharpoonup(\tau_1,\mu_1)\rightharpoonup(\tau_2,\mu_2)\rightharpoonup\cdots$ on $S$ related by maximal splittings, with some period $n=n(f)\ge1$ so that $f_*(\tau_m,\mu_m)=\lambda(\tau_{m+n},\mu_{m+n})$ for all $m$, where $\lambda=\lambda(f)>0$ is the dilatation of $f$. The periodic splitting sequence $\{(\tau_m,\mu_m)\}_m$ is well-defined up to scaling and index-shifting. If $c$ is any mapping class that commutes with $f$, orientation-preserving or not, then $\{c_*(\tau_m,\mu_m)\}_m$ is another periodic splitting sequence for $f$, hence by uniqueness we must have $c_*(\tau_m,\mu_m)=r(\tau_{m+k},\mu_{m+k})$ for all $m$, for some $r=r(c)>0$, $k=k(c)\in\Z$.

If $a^2=f$, then $a$ commutes with $f$, and we see from the discussion above that $$\lambda(\tau_n,\mu_n)=f_*(\tau_0,\mu_0)=(a_*)^2(\tau_0,\mu_0)=r(a)^2(\tau_{2k(a)},\mu_{2k(a)}).$$
Since different $(\tau_m,\mu_m)$'s are never equal up to scaling, we must have $k(a)=n(f)/2$, $r(a)=\sqrt{\lambda(f)}$. If $n(f)$ is odd, such $a$ does not exist; if $n(f)$ is even, one can find all potential such $a$ by solving $a_*(\tau_0,\mu_0)=\sqrt{\lambda(f)}(\tau_{n(f)/2},\mu_{n(f)/2})$. Then one checks for genuine solutions among this list.
\end{proof}

Since $\phi^M$ restricts to $\Sigma_1\cup\Sigma_2$, $\Phi^M$ restricts to a map on the product region $P_0$ exchanging $\Sigma_1^+$ and $\Sigma_2^+$. It follows that $\phi^M|_{\Sigma_1^+}$ is isotopic to $\theta\circ\phi^M|_{\Sigma_2^+}\circ\theta$ rel the fixed boundaries. Using $\phi^M|_{\Sigma_2^+}$ to identify $\Sigma_2^+=\Sigma_1^+$, this can be rephrased as saying that the orientation-reversing homeomorphism $\theta\colon\Sigma_1^+\to\Sigma_1^+$ satisfies that $\theta^2$ is isotopic to $\phi^M|_{\Sigma_1^+}$ up to isotopy rel the fixed boundaries. By Lemma~\ref{lem:square_root}, the number of such $\theta$ is finite up to isotopy not necessarily rel boundary.

Suppose $\theta,\theta'\colon\Sigma_1^+\to\Sigma_2^+$ are maps rel the fixed boundaries which are isotopic (not necessarily rel boundary), that give rise to compressions $C,C'$ of $\phi$, respectively. We claim that $C$ and $C'$ are related by a symmetry of $(S,\phi)$.

After identifying $\Sigma_2^+=\Sigma_1^+$ via $\phi^M|_{\Sigma_2^+}$, each of $\theta,\theta'$ permutes the fixed boundary components of $\Sigma_1^+$ cyclically, according to $\phi^M$. By assumption, $\theta'\colon\Sigma_1^+\to\Sigma_1^+$ is isotopic rel the fixed boundaries to some $\theta\prod_{j=1}^kT_j^{n_j}$, where $k$ is the number of fixed boundaries of $\Sigma_1^+$, and $T_j$ is the boundary Dehn twist along the $j$-th fixed boundary component of $\Sigma_1^+$, where the boundaries are ordered cyclically with respect to $\phi^M$. Since $\theta$ is orientation-reversing, $\theta\circ T_j\circ\theta^{-1}$ is isotopic to $T_{j+1}^{-1}$ rel the fixed boundaries. It follows that $(\theta')^2$ is isotopic rel the fixed boundaries to $\theta^2\prod_{j=1}^kT_j^{n_j-n_{j+1}}$. By assumption, this implies that $n_j=n$ is independent of $j$. One may now check that the multi-twist $c:=\prod_{i=0}^{M-1}(\phi^i\circ(\prod_{j=1}^kT_j)^n\circ\phi^{-i})$ is a symmetry of $(S,\phi)$ relating $C'$ to $C$, proving the claim.

\textbf{Case 2.2}: $\Sigma_1=\Sigma_2$.

This is analogous to Case 2.1.2, with minor changes. We repeat the argument as follows.

As before, for large $q$, $\gamma$ and $\gamma_q$ fill $\Sigma_1$. The union of closed faces of $D\cup D'$ adjacent to an interior disk component $D_0$ of the complement of $\gamma\cup\gamma_q$ in $\Sigma_1$ is an annulus with boundary $\partial D_0\cup\gamma'$ for some trivial loop $\gamma'$ (it cannot be a M\"obius band with boundary $\partial D_0$, since otherwise we find an embedded $\mathbb{RP}^2$ in $C$, contradicting that compression bodies are $\mathbb{RP}^2$-irreducible). Now, we see some sphere components of the interior boundary of $\nu(S\cup D\cup D')$, each of which contains two interior disk components of the complement of $\gamma\cup\gamma_q$ in $\Sigma_1$. There are four other sphere components, each of which contains two boundary disk components of the complement of $\gamma\cup\gamma_q$ in $\Sigma_1$, one on each side of $\delta_1$. Filling in all these spheres by balls, we obtain a twisted $I$-bundle region bounding $\Sigma_1$, whose vertical annulus boundary component corresponding to $\delta_1$ is glued to $\nu(\delta_1)$. This product region gives rise to a compression body $C_{00}\subset C$ with exterior boundary $S$.

The union $C'':=\cup_{i\ge0}\Phi^i(C_{00})\subset C$ is a compression body with exterior boundary $S$ that consists of disjoint twisted $I$-bundle regions $P_i$ bounding $\phi^i(\Sigma_1)$, $i=0,1,\cdots,M-1$, where $M$ is the minimal integer with $\phi^M(\Sigma_1)=\Sigma_1$. An annulus boundary of a product region $P_i$ adjacent to some $\phi^j(\nu(\delta_1))$ is glued to such $\phi^j(\nu(\delta_1))$. Again we have $C''=C$ by minimality of $C$.

The compression body $C$ is determined by the product region $P_0$, which is in turn determined by the orientation-reversing free involution $\iota\colon\Sigma_1^+\to\Sigma_1^+$ induced by $P_0$, where $\Sigma_1^+$ is the union of $\Sigma_1$ and boundary annuli in $S\backslash int(\Sigma_1)$ bounded by $\cup_{j\ge0}\phi^j(\delta_1)$, but with each $\phi^j(\delta_1)$ replaced by two copies of itself, each glued to the annulus on one side. Thus, the map $\Sigma_1^+\to S$ is $2$ to $1$ over each $\phi^j(\delta_1)$, and an embedding elsewhere. A boundary of $\Sigma_1^+$ is a \textit{fixed boundary} if it is a copy of some $\phi^j(\delta_1)$, and is a \textit{free boundary} otherwise. The involution $\iota$ is rel the fixed boundaries of $\Sigma_1^+$ (on which there is a canonical involution exchanging copies of $\phi^j(\delta_1)$), and is considered up to conjugations by isotopies rel the fixed boundaries. It remains to show that for any fixed $\Sigma_1,\delta_1$, the number of such $\iota\colon\Sigma_1^+\to\Sigma_1^+$ is finite up to symmetries.

We prove the following lemma.

\begin{Lem}\label{lem:involutive_conjugator}
Let $S$ be a connected oriented hyperbolic surface without boundary, with (possibly empty) punctures, and $f$ be a pseudo-Anosov mapping class on $S$. There is an algorithm to output all mapping classes $a$ on $S$, orientation-preserving or not, with $a^2=[f,a]=\mathrm{id}$. In particular, there are at most finitely many such mapping classes $a$.
\end{Lem}
\begin{proof}
We proceed as in Lemma~\ref{lem:square_root}. Let $\{(\tau_m,\mu_m)\}_m$ be the periodic splitting sequence of the stable lamination of $f$. All $\iota$ as in the lemma are found by solving $\iota_*(\tau_0,\mu_0)=(\tau_0,\mu_0)$ and re-checking the conditions.
\end{proof}

Since $\phi^M$ restricts to $\Sigma_1$, $\Phi^M$ restricts to the twisted product region $P_0$. It follows that $\phi^M|_{\Sigma_1^+}$ is isotopic to $\iota^{-1}\circ\phi^M|_{\Sigma_1^+}\circ\iota$ rel the fixed boundaries. By Lemma~\ref{lem:involutive_conjugator}, the number of such involutive $\iota$ is finite up to isotopy not necessarily rel boundary.

Suppose $\iota,\iota'\colon\Sigma_1^+\to\Sigma_1^+$ are maps rel the fixed boundaries which are isotopic (not necessarily rel boundary), that give rise to compressions $C,C'$ of $\phi$, respectively. We claim that $C$ and $C'$ are related by a symmetry of $(S,\phi)$.

By assumption, $\iota'$ is isotopic to some $\iota\prod_{j=1}^kT_j^{n_j}$ rel the fixed boundaries, where $k$ is the number of fixed boundary components of $\Sigma_1^+$ and $T_j$ is the boundary Dehn twist along the $j$-th boundary component. By $(\iota')^2=\iota^2=\mathrm{id}_{\Sigma_1^+}$ and the fact that $\iota,\iota'$ each commute with $\phi^M$ up to isotopy rel the fixed boundaries, an argument analogous to the one in Case 2.1.2 shows that $n_j=n$ is independent of $j$. One may now check that the multi-twist $c:=\prod_{i=0}^{M-1}(\phi^i\circ(\prod_{j=1}^kT_j)^n\circ\phi^{-i})$ is a symmetry of $(S,\phi)$ relating $C'$ to $C$, proving the claim.

\textbf{Case 3}: $|\gamma\cap\delta|=4$.

See Figure~\ref{fig:D}. Applying a further cut-and-paste operation if necessary, we may assume that $\delta_1=\delta_2'$, $\delta_2=\delta_1'$, $\Sigma_1'=\Sigma_1''$, and that $\gamma\cup\gamma_q$ intersects $\Sigma_1'$ in $4$ parallel essential arcs $\sigma_1,\cdots,\sigma_4$ in that order, each parallel to some $\sigma$. We exhibit an essential curve $\gamma'$ disjoint from $\delta$ that compresses in $C$, contradicting the minimality of $|\gamma\cap\delta|$.

\textbf{Case 3.1}: $\Sigma_1\ne\Sigma_2$.

As in Case 2.1, for large $q$, $\gamma$ and $\gamma_q$ fill $\Sigma_i$, $i=1,2$, and $D\cup D'$ is homeomorphic to $I\times((\gamma\cup\gamma_q)\cap\Sigma_1)$. On the interior boundary of $\nu(S\cup D\cup D')$, there are some sphere components that each contain one interior disk on $\Sigma_1$ and one on $\Sigma_2$, and three other sphere components that each contain one boundary disk component on $\Sigma_1$, one on $\Sigma_2$, and the annulus region on $\Sigma_1'$ bounded by $\sigma_i$ and $\sigma_{i+1}$, for some $i\in\{1,2,3\}$. Filling in these spheres gives a compression body $C_{00}\subset C$ with exterior boundary $S$, which consists of a product region cobounding $\Sigma_1,\Sigma_2$, and an extra $2$-handle attached along the union of $\sigma$ and an arc on the vertical boundary of the product region.

We see that the essential simple closed curve $\gamma':=\sigma_+\cup\delta_{1,0}\cup\sigma_-\cup\delta_{2,0}$ is a curve essentially contained in $\Sigma_1'$ that compresses in $C$, where $\sigma_\pm$ are two parallel copies of $\sigma$, and $\delta_{i,0}\subset\delta_i$ is the longer subarc on $\delta_i$ connecting the endpoints of $\sigma_+$ and $\sigma_-$ on $\delta_i$, $i=1,2$.

\textbf{Case 3.2}: $\Sigma_1=\Sigma_2$.

As in Case 2.2, for large $q$, $\gamma$ and $\gamma_q$ fill $\Sigma_1$. Filling in sphere components of $\nu(S\cup D\cup D')$ gives a compression body $C_{00}\subset C$ with exterior boundary $S$ that consists of a twisted $I$-bundle region bounding $\Sigma_1$ together with a $2$-handle attached along the union of $\sigma$ and an arc on the vertical boundary of the twisted product (in particular, $\delta_1\ne\delta_2$). The essential simple closed curve $\gamma'$ given exactly in the same way as in Case 3.1 is essentially contained in $\Sigma_1'$, and compresses in $C$.
\end{proof}

\subsection{Canonical forms of minimal compressions}\label{sbsec:comp_forms}
By the proof in Section~\ref{sbsec:comp_finite}, we see that every minimal compression of a surface homeomorphism $\phi$ of a hyperbolic surface $S$ (with canonical reduction system $\delta\subset S$) has one of the following forms:

\textbf{1.1} The compression that compresses every component of $\cup_{i\ge0}\phi^i(\gamma)$, for some component $\gamma$ of $\delta\cup\partial S$ for which no planar component of $S\backslash\nu(\delta)$ has all but exactly one boundary component contained in $\cup_{i\ge0}\phi^i(\overline{\nu(\gamma)})$. (In the descriptions of compressions henceforth, spherical components are always assumed to be capped off by balls.)

This corresponds to \textbf{Case 1.1} in Section~\ref{sbsec:comp_finite}. One may check that the extra assumption regarding planar components is a necessary and sufficient condition for the compression to be minimal. Sufficiency is clear. To show necessity, observe that if some planar component of $S\backslash\nu(\delta)$ violates the condition, having a single boundary component corresponding to $\gamma'\subset\delta\cup\partial S$ not contained in $\cup_{i\ge0}\phi^i(\gamma)$, then either each $\phi^i(\gamma)$ is adjacent to the same planar piece of $S\backslash\nu(\delta)$ on both sides, in which case the compression determined by $\gamma'$ gives a proper subcompression, or $\#\pi_0(\cup_{i\ge0}\phi^i(\gamma))>\#\pi_0(\cup_{i\ge0}\phi^i(\gamma'))$, in which case we can replace $\gamma$ by $\gamma'$ and induct to show the non-minimality.

\textbf{1.2} The compression that compresses every component of $\cup_{i\ge0}\phi^i(\gamma)$, for some essential (not $\partial$-parallel) simple closed curve $\gamma$ contained in a periodic piece of $S$ such that
\begin{itemize}
\item each $\phi^i(\gamma)$ is either disjoint from $\gamma$ or identical to $\gamma$;
\item there does not exist a nontrivial simple closed curve $\gamma'\subset S$ (possibly $\partial$-parallel) disjoint from $\cup_{i\ge0}\phi^i(\gamma)$, with each $\phi^i(\gamma')$ either disjoint from $\gamma'$ or identical to $\gamma'$, such that there is a planar component of $S\backslash\nu(\cup_{i\ge0}(\gamma\cup\gamma'))$ that has one component contained in $\nu(\cup_{i\ge0}\phi^i(\gamma'))$ and all other components contained in $\nu(\cup_{i\ge0}\phi^i(\gamma))$.
\end{itemize}

This corresponds to \textbf{Case 1.2} in Section~\ref{sbsec:comp_finite}. We claim that the extra assumption regarding the nonexistence of $\gamma'$ is a necessary and sufficient condition for the compression to be minimal.

Let $S_i$ denote the periodic piece being compressed, and $C_i(\gamma)$ denote the corresponding compression for $\phi|_{S_i}$. To show the claim, we look at the picture in the quotient orbifold $\overline{S_i}=S_i/\phi|_{S_i}$. Let $\bar\gamma\subset\overline{S_i}$ denote the image of $\gamma$. The cyclic covering induces a map $h\colon H_1(\overline{S_i})\to\Z/N$, where $N\ge1$ is the covering degree. We attach a $2$-handle to the interior boundary of $I\times\overline{S_i}$ along $\bar\gamma$ and assign along the cocore an arc of cyclic singularities of order $N/(gcd(h([\bar\gamma]),N))$ if $h([\bar\gamma])\ne0$, so that $h$ extends (uniquely) to the resulting $3$-orbifold, denoted $\overline{C_i}'$. If $h([\bar\gamma])=0$ and $\bar\gamma$ bounds a disk containing two singularities (of necessarily equal order), define $\overline{C_i}$ to be the union of $\overline{C_i}'$ with balls with singular arcs connecting the singularities on the boundaries, one for each such disk bounding $\bar\gamma$; if $h([\bar\gamma])\ne0$, define $\overline{C_i}=\overline{C_i}'$. The $3$-orbifold $\overline{C_i}=\overline{C_i}(\bar\gamma)$ is regarded as a compression of $(\overline{S_i},h)$, determined by $\bar\gamma$, and its $N$-fold cyclic cover determined by the unique extension of $h$ is the compression $C_i(\gamma)$ of $\phi|_{S_i}$. Conversely, the quotient of $C_i(\gamma)$ by any periodic extension of $\phi|_{S_i}$ is homeomorphic to $\overline{C_i}(\bar\gamma)$.

Suppose $\bar\gamma'\subset\overline{S_i}$ is another simple closed curve not bounding a disk containing at most one orbifold singularity, and $\gamma'\subset S_i$ one of its lifts. Then one can show from \cite[Lemma~4.1]{bonahon1983cobordism} that $C_i(\gamma')$ is a subcompression of $C_i(\gamma)$ if and only if $\overline{C_i}(\bar\gamma')$ is a subcompression of $\overline{C_i}(\bar\gamma)$. The latter happens if and only if $h([\bar\gamma])=0$ and $\bar\gamma'$ (after an isotopy) cobound with $\bar\gamma$ an annulus with exactly one orbifold singularity. But the (non)existence of such $\bar\gamma'$ is equivalent to the (non)existence of $\gamma'$ in the assumption, hence the claim.

\textbf{1.3} A compression obtained by gluing a minimal compression of a pseudo-Anosov piece (not compressing any boundary of the piece) of $S$ to the identity compression of the complement of this piece.

This corresponds to \textbf{Case 1.3} in Section~\ref{sbsec:comp_finite}. The stronger conclusion stated here, namely that any minimal compression that compresses an essential curve in a pseudo-Anosov piece is supported on that piece, can be deduced from the proof of Theorem~2.1 in \cite{casson1985algorithmic}.

\textbf{2.1.1} A compression $C$ as described in \textbf{Case 2.1.1} in Section~\ref{sbsec:comp_finite}.

More precisely, let $S_i=\sqcup_{j=1}^M\Sigma_{i,j}$ be a pseudo-Anosov piece of $S$, where the $\Sigma_{i,j}$'s are connected components of $S_i$ that are permuted cyclically under $\phi$, $i=1,2$, $S_1\ne S_2$. Let $\delta_1$ be a component of $\delta$ with $\nu(\delta_1)$ adjacent to both $\Sigma_{1,1}$ and $\Sigma_{2,1}$. Let $S_i^+$ denote the union of $S_i$ with boundary annuli in $S\backslash int(S_i)$ bounded by $\cup_{j\ge0}\phi^j(\delta_1)$, $i=1,2$. Let $\theta\colon S_1^+\to S_2^+$ be an orientation-reversing homeomorphism rel the identity map on the common boundaries (called the \textit{fixed boundaries}), which intertwines $\phi|_{S_1^+}$ and $\phi|_{S_2^+}$ up to isotopy rel the fixed boundaries (note that necessarily $\theta$ maps $\Sigma_{1,j}$ to $\Sigma_{2,j}$ for each $j$). Then $\theta$ determines a compression $C$ of $\phi$ consisting of a product region cobounding $S_1^+$ and $S_2^+$, inside which $S_1^+$ and $S_2^+$ are identified via $\theta$.

\textbf{2.1.2} A compression $C$ as described in \textbf{Case 2.1.2} in Section~\ref{sbsec:comp_finite}.

More precisely, let $S_1=\sqcup_{j=1}^{2M}\Sigma_{1,j}$ be a pseudo-Anosov piece of $S$, where the $\Sigma_{1,j}$'s are connected components of $S_1$ that are permuted cyclically under $\phi$. Let $\delta_1$ be a component of $\delta$ with $\nu(\delta_1)$ adjacent to both $\Sigma_{1,1}$ and $\Sigma_{1,M+1}$. Let $S_1^+$ denote the union of $S_1$ with boundary annuli in $S\backslash int(S_1)$ bounded by $\cup_{j\ge0}(\phi^j(\delta_1))$, but with each $\phi^j(\delta_1)$ replaced by two copies, each glued to the annulus on one side. Let $\theta\colon S_1^+\to S_1^+$ denote an orientation-reversing involution rel the natural involution on the union of boundary components arising from copies of $\phi^j(\delta_1)$ (called the \textit{fixed boundaries}), that commutes with $\phi|_{S_1^+}$ up to isotopy rel the fixed boundaries (note that necessarily $\theta$ maps $\Sigma_{1,j}$ to $\Sigma_{1,M+j}$ for each $j=1,\cdots,M$). Then $\theta$ determines a compression $C$ of $\phi$ consisting of a twisted product region bounding $S_1^+$, inside which $S_1^+$ is identified with itself via $\theta$. The twisted product region is the disjoint union of $M$ connected product regions.

\textbf{2.2} A compression $C$ as described in \textbf{Case 2.2} in Section~\ref{sbsec:comp_finite}.

More precisely, let $S_1=\sqcup_{j=1}^M\Sigma_{1,j}$ be a pseudo-Anosov piece of $S$, where the $\Sigma_{1,j}$'s are connected components of $S_1$ that are permuted cyclically under $\phi$. Let $\delta_1$ be a component of $\delta$ with $\nu(\delta_1)$ adjacent to $\Sigma_{1,1}$ on both sides. Let $S_1^+$ and its fixed boundaries be defined in the same way as in \textbf{2.1.2}, and let $\iota\colon S_1^+\to S_1^+$ be an orientation-reversing free involution rel the natural involution on the fixed boundaries that commutes with $\phi|_{S_1^+}$ up to isotopy rel the fixed boundaries (note that necessarily $\iota$ maps $\Sigma_{1,j}$ to $\Sigma_{1,j}$ for each $j$). Then $\iota$ determines a compression $C$ of $\phi$ consisting of a twisted product region bounding $S_1^+$. The twisted product region is a disjoint union of $M$ connected twisted product regions.

\medskip

Conversely, one may check that every compression of $\phi$ of the form listed in one of the six cases above is minimal, and that the six cases are mutually exclusive. Moreover, a symmetry of $(S,\phi)$ does not change the form of a minimal compression from one to another.\medskip

We note that the forms \textbf{2.1.1, 2.1.2, 2.2} have a uniform description as follows. Let $\delta_1\subset\delta$ be a component of $\delta$. Assume the (one or two) pieces of $S$ adjacent to $\nu(\delta_1)$ are all pseudo-Anosov, and let $S_{\delta_1}$ be their union. Let $S_{\delta_1}^+$ denote the union of $S_{\delta_1}$ with $\cup_{j\ge0}\overline{\nu(\phi^j(\delta_1))}$, but with each $\phi^j(\delta_1)$ replaced by two copies, each glued to one side of its complement. A boundary component of $S_{\delta_1}^+$ is called a \textit{fixed boundary} if it arises from some $\phi^j(\delta_1)$. Let $\iota\colon S_{\delta_1}^+\to S_{\delta_1}^+$ be an orientation-reversing free involution rel the natural involution on the fixed boundaries, which commutes with $\phi_1:=\phi|_{S_{\delta_1}^+}$ up to isotopy rel the fixed boundaries. Then, $\iota$ determines an $I$-bundle region bounding $S_{\delta_1}^+$ which gives rise to a minimal compression $C=C_\iota$ of $\phi$.

For later purposes, we note the following constraint on the $\phi_1=\phi|_{S_{\delta_1}^+}$ in this unified setup for \textbf{2.1.1, 2.1.2, 2.2} which is necessary for the existence of $\iota$. Let $N>0$ be chosen so that $\phi^N$ preserves $\delta_1$ with orientation. After an isotopy, we may assume $\phi^N$ fixes $\delta_1$ pointwise. Then, the sum of the fractional Dehn twist coefficients of $\phi_1^N$ about the two boundaries arising from $\delta_1$ is zero. This follows from the fact that $\iota\phi_1^N\iota^{-1}$ is isotopic to $\phi_1^N$ rel the fixed boundaries, and that $\iota$ is orientation-reversing, exchanging the fixed boundaries in pairs.

We refer the reader to Section~\ref{sbsec:C21(4_1)} for a motivating example.\smallskip

We end this section with the following lemma for the next section, which removes the ambiguity on the condition $[\iota,\phi_1]\sim\mathrm{id}_{S_{\delta_1}^+}$ between free isotopies and isotopies rel the fixed boundaries.

\begin{Lem}\label{lem:FDTC_sufficient}
Let $S_{\delta_1^+}$ and $\phi_1$ be defined as above, so that the fractional Dehn twist coefficient condition on $\phi_1^N$ holds. If $\iota$ is an orientation-reversing free involution on $S_{\delta_1^+}$ that commutes with $\phi_1$ up to isotopy not necessarily rel boundary, then there exists an orientation-reversing free involution $\iota'$ freely isotopic to $\iota$ that commutes with $\phi_1$ up to isotopy rel the fixed boundaries. Such an $\iota'$ is computable from the data.
\end{Lem}
\begin{proof}
Let $\sim$ denote the equivalence relation given by isotopies rel the fixed boundaries. By assumption, $\iota\phi_1\iota^{-1}\sim\phi_1\prod_{j=1}^{2k}T_j^{n_j}$ for some integers $n_j$, where $k=\#\pi_0(\cup_{j\ge0}\phi^j(\delta_1))$, and $T_j$ is the boundary Dehn twist along the $j$-th boundary component of $S_{\delta_1}^+$. The boundaries of $S_{\delta_1}^+$ are assumed to be ordered so that the $j$-th and $(k+j)$-th boundary components are the ones corresponding to $\phi^j(\delta_1)$, and that $\phi_1$ either permutes all components cyclically in order, or permutes the first and last $k$ components cyclically in order respectively.

Since $\iota$ and $\phi_1\iota\phi_1^{-1}$ are both involutions, we must have $n_j=n_{k+j}$ for all $j=1,\cdots,k$. Consequently, $\iota\phi_1^N\iota^{-1}\sim\phi_1^N(\prod_{j=1}^{2k}T_j)^{NS/k}$, where $S=\sum_{j=1}^kn_j=\sum_{j=k+1}^{2k}n_j$. Let $\partial_j$ denote the $j$-th boundary component of $S_{\delta_1}^+$, we have $$FDTC(\phi_1^N,\partial_{k+1})+NS/k=FDTC(\iota\phi_1^N\iota^{-1},\partial_{k+1})=-FDTC(\phi_1^N,\partial_1)=FDTC(\phi_1^N,\partial_{k+1}),$$ implying that $S=0$.

Now, define $\iota'=\iota\prod_{j=1}^{2k}T_j^{r_j}$, where $r_j=\sum_{s=1}^{j-1}n_s$. Then $\iota'$ is a free involution since $r_j=r_{k+j}$ for all $j=1,\cdots,k$. Moreover, $$[\phi_1,\iota']\sim\phi_1\iota(\prod_jT_j^{r_j})\phi_1^{-1}(\prod_jT_j^{-r_j})\iota^{-1}\sim\phi_1\iota\phi_1^{-1}(\prod_jT_j^{r_{j-1}-r_j})\iota^{-1}\sim\iota(\prod_jT_j^{n_{j-1}+r_{j-1}-r_j})\iota^{-1}\sim\mathrm{id},$$ as desired.
\end{proof}

\subsection{Algorithm for minimal compressions}\label{sbsec:comp_algo}
In addition to Lemma~\ref{lem:square_root} and Lemma~\ref{lem:involutive_conjugator}, we shall make use of the following two results in our algorithm.
\begin{Lem}[Conjugacy Problem, \cite{hemion1979classification}]\label{lem:conjugacy}
Given two homeomorphisms $f,g\colon\Sigma\to\Sigma$ on a compact oriented surface $\Sigma$, it is decidable whether there is a homeomorphism $h\colon\Sigma\to\Sigma$ inducing a given permutation $\pi_0(\partial\Sigma)\to\pi_0(\partial\Sigma)$ such that $hfh^{-1}$ is isotopic to $g$ (not necessarily rel boundary). One may compute one such $h$ if it exists.
\end{Lem}
\begin{Lem}[Centralizer, \cite{rafi2020centralizers}]\label{lem:centralizer}
Given a homeomorphism $f\colon\Sigma\to\Sigma$ on a compact oriented surface $\Sigma$, there is an algorithm to find a generating set of $C(\Sigma,f)$, the symmetry group of $(\Sigma,f)$.
\end{Lem}
\cite{rafi2020centralizers} stated Lemma~\ref{lem:centralizer} for mapping class groups with free boundary condition, but the version here with fixed boundary condition can be deduced as a consequence by analyzing the permutation of boundary components and boundary twists and applying Schreier's lemma for subgroups, for instance.\smallskip

\begin{proof}[Proof of Theorem~\ref{thm:compression}(2)]
Before giving the algorithm, we make the following observation. Any symmetry of $(S,\phi)$ preserves the canonical reduction system $\delta$ of $\phi$, and permutes periodic (resp. pseudo-Anosov) pieces of $S$. This information can be encoded by a map $C(S,\phi)\to\mathrm{Aut}(\pi_0(\partial(\nu(\delta))))$, whose image is computable, since it is a finite group, and a generating set of it can be written down thanks to Lemma~\ref{lem:centralizer}. In particular, we can decide which components of $\delta$ can be related to one another, which complementary regions of $\delta$ in $S$ can be related to one another, and how the boundaries of the complementary regions can be related to one another, by symmetries of $(S,\phi)$.

Now, for each of the six canonical forms (\textbf{1.1}, \textbf{1.2}, \textbf{1.3}, \textbf{2.1.1}, \textbf{2.1.2}, \textbf{2.2}) listed in Section~\ref{sbsec:comp_forms}, we give an algorithm that outputs all minimal compressions of that form up to symmetry. For each case, the algorithm consists of three steps:
\begin{enumerate}[1.]
\item Produce a list of compressions of $\phi$ of the given form, denoted $\mathcal C=\{C_1,\cdots,C_N\}$, such that each minimal compression of $\phi$ of the given form is among these up to symmetry.
\item Throw away compressions in the list $\mathcal C$ that are not minimal and obtain a new list $\mathcal C'\subset\mathcal C$.
\item Remove duplicates from the list $\mathcal C'$; that is, among all minimal compressions related by symmetries, keep only one representative. Output the new list $\mathcal C''\subset\mathcal C'$ that we obtain.
\end{enumerate}

\textbf{1.1} Step 1: For each component $\gamma$ of $\delta\cup\partial S$ satisfying the condition about planar components, create the compression body $C_\gamma$ that compresses every component of $\cup_{i\ge0}\phi^i(\gamma)$, then output all these $C_\gamma$ as $\mathcal C$.

Step 2: Discard from $\mathcal C$ the compressions $C_\gamma$ for which $\gamma$ violates the assumption on planar components in the description of \textbf{1.1} in Section~\ref{sbsec:comp_forms} and obtain $\mathcal C'$.

Step 3: $C_{\gamma_1}$ is related to $C_{\gamma_2}$ by a symmetry if and only if the $\phi$-orbit of $\gamma_1$ is related to the $\phi$-orbit of $\gamma_2$ by a symmetry of $(S,\phi)$. This can be determined thanks to the observation at the beginning of the proof, enabling us to remove duplicates from $\mathcal C'$ and obtain $\mathcal C''$.

\textbf{1.2} Step 1: For every periodic piece $S_i$ of $S$, construct the quotient orbifold $S_i\to\overline{S_i}=S_i/(\phi|_{S_i})$. Let $G_2<G_1<G_0=\mathrm{Homeo}_\partial^+(\overline{S_i})$ be groups defined as in \textbf{Case 1.2} in Section~\ref{sbsec:comp_finite}. Write a list $\{\overline{\gamma_1,}\cdots,\overline{\gamma_r}\}$ of all essential simple closed curves in $\overline{S_i}$ up to $G_0$ (this is a finite process). Since one can write down a finite generating set for the rel boundary mapping class group of $\overline{S_i}$, by definition of $G_2$, one can write down a finite set $\{\overline{\phi_1},\cdots,\overline{\phi_s}\}$ of elements in $G_0$ that contains one element in each right coset of $G_2$ in $G_0$. Then, $\{\overline{\phi_j}(\overline{\gamma_k})\}$ gives a list of simple closed curves in $\overline{S_i}$ up to $G_2$. The lifts of these curves in $S_i$, for various $i$, give a list $\mathcal C$ of minimal compressions of $\phi$, exhaustive up to symmetry.

Step 2: By the assumption on the nonexistence of $\gamma'$ imposed in the description of \textbf{1.2} in Section~\ref{sbsec:comp_forms} and the subsequent description in the quotient orbifold $\overline{S_i}$, a curve $\bar\gamma\subset\overline{S_i}$ determines a minimal compression of $\phi$ if and only if $h([\bar\gamma])\ne0$, or $h([\bar\gamma])=0$ and every orbifold singularity of $\overline{S_i}$ is contained in a disk bounded by $\gamma$ containing two orbifold singularities. We throw away compressions in $\mathcal C$ determined by curves in $\overline{S_i}$ violating this condition and obtain $\mathcal C'$.

Step 3: By the observation at the beginning of the proof, we can determine which periodic pieces are related to one another by symmetries. Pick one representative for each equivalence class of periodic pieces of $S$ under symmetries, and throw away minimal compressions in $\mathcal C'$ that compress curves in a periodic piece that is not chosen as a representative. Denote by $\mathcal C_0''$ the resulting list of minimal compressions.

We still need to remove from $\mathcal C_0''$, for each periodic piece $S_i$ chosen as a representative, duplicates among compressions that compress curves in $S_i$ that are related by symmetries of $(S,\phi)$ which does not come from symmetries of $(S_i,\phi|_{S_i})$. If $c\colon S\to S$ is a symmetry of $(S,\phi)$ that restricts to a homeomorphism $S_i\to S_i$ which permutes boundary components of $S_i$ trivially, we can assume, by isotoping $c$ rel boundary and applying Lemma~\ref{lem:commutator_no_boundary_twist} as in the argument of \textbf{Case 2.1.1} in Section~\ref{sbsec:comp_finite} if necessary, that $c|_{S_i}$ commutes with $\phi|_{S_i}$ up to isotopy rel boundary, providing a symmetry of $(S_i,\phi|_{S_i})$. Hence we only need to account for symmetries $c\colon S\to S$ that restricts to $S_i$ and that permutes the boundary components of $S_i$ nontrivially. By the observation at the beginning of the proof, we can understand how symmetries of $(S,\phi)$ preserving $S_i$ can permute boundary components of $S_i$. This enables us to further check and remove duplicates in $\mathcal C_0''$ and obtain $\mathcal C''$ as desired.

\textbf{1.3} Step 1: Fix complete hyperbolic metrics on the interior of pseudo-Anosov pieces of $S$. By the proof in \cite{casson1985algorithmic}, the minimal length of an essential simple closed curve that compresses in a nontrivial compression of $\phi$ is bounded above by a computable constant independent of the compression. We enumerate all essential simple closed curves in pseudo-Anosov pieces of $S$ with length bounded by this bound. For each such curve $\gamma$, apply the strategy in the proof of Theorem~2.1 in \cite{casson1985algorithmic} to construct a list of compressions that includes all minimal compressions containing $\gamma$. Together, this gives a list $\mathcal C$ as desired. Here, a compression of $\phi$ is presented by a set of simple closed curves in $S$ so that compressing all the curves in the set (and capping off spheres if any) yields the compression body.

Step 2,3: We perform these two steps together. Since minimal compressions in the list $\mathcal C$ are exhaustive up to symmetry, we know that every non-minimal compression in $\mathcal C$ properly contains another compression in $\mathcal C$ after modifying by a symmetry. We claim that, for any compression $C\in\mathcal C$, it is possible to write down all (finitely many) compressions of $\phi$ that are related to $C$ by a symmetry. Since it is easy to determine the containment relation between two compression bodies with exterior boundary $S$, this would enable us to remove non-minimal compressions as well as duplicates in $\mathcal C$.

Let $C\in\mathcal C$ be any compression that compresses curves in some pseudo-Anosov piece $S_i$ of $S$. Every symmetry of $(S_i,\phi|_{S_i})$ extends to a symmetry of $(S,\phi)$ by identity outside $S_i$. By the observation at the beginning of the proof, we can determine all possible images of $S_i$ under a symmetry, as well as, for each such image $S_j$, all possible bijections $p\colon\pi_0(\partial S_i)\to\pi_0(\partial S_j)$ that such symmetries can give rise to. In addition, we can pick an explicit symmetry $c_{j,p}$ realizing each given realizable $j$ and $p\colon\pi_0(\partial S_i)\to\pi_0(\partial S_j)$. If $c'$ is another symmetry realizing $(j,p)$, then $c'$ and $c$ differ on $S_i$ by pre-composing a symmetry of $(S_i,\phi|_{S_i})$. Pick $N>0$ so that $\phi|_{S_i}^N$ permutes the boundary components of $S_i$ trivially and has integral fractional Dehn twist coefficients about each of them. Then, one may straighten $\phi|_{S_i}^N$ to some $\psi$ rel boundary which provides a symmetry of $(S_i,\phi|_{S_i})$. By the proof of \cite[Proposition~3.6]{rafi2020centralizers} or that of Lemma~\ref{lem:square_root} or \ref{lem:involutive_conjugator}, one can write down explicitly a finite set $\{c_1,\cdots,c_K\}$ of symmetries of $(S_i,\phi|_{S_i})$ that contains at least one representative for each $\langle\psi\rangle$-coset in $C(S_i,\phi|_{S_i})$ up to boundary twists. Now, since $\psi$ and boundary twists act trivially on $C$, the finite set $\{(c_{j,p}\circ c_k)_*C\}$ contains all compressions of $\phi$ that are related to $C$ by a symmetry.

\textbf{2.1.1} Step 1: For any component $\delta_1\subset\delta$ adjacent to two distinct pieces $S_1,S_2$ of $S$ (ordered in an arbitrary way), check whether $S_1,S_2$ have the same topological type, and whether they are both pseudo-Anosov pieces. If so, we define $S_1^+,S_2^+$ as before, and proceed to find possible product regions cobounding $S_1^+$ and $S_2^+$. We first pick $kM>0$ with $\phi^{kM}$ preserving $\delta_1$, and check if the sum of the fractional Dehn twist coefficients of $\phi^{kM}|_{S_1^+}$ and $\phi^{kM}|_{S_2^+}$ about $\delta_1$ is zero. If so, for any bijection $p\colon\pi_0(S_1^+)\to\pi_0(S_2^+)$ respecting the natural identification of the common boundaries, Lemma~\ref{lem:conjugacy} enables us to determine whether there is an orientation-reversing homeomorphism $\theta\colon S_1^+\to S_2^+$ rel the fixed boundaries inducing $p$ that intertwines $\phi|_{S_i^+}$, $i=1,2$, up to isotopy not necessarily rel boundary. If such a $\theta$ is found, by Lemma~\ref{lem:FDTC_sufficient}, we may modify $\theta$ by boundary Dehn twists so that $\theta$ intertwines $\phi|_{S_i^+}$, $i=1,2$, up to isotopy rel the fixed boundaries. This new $\theta$ gives rise to a compression $C_{\delta_1,p}$ of $\phi$ as described in Section~\ref{sbsec:comp_forms}. The argument in \textbf{Case 2.1.1} in Section~\ref{sbsec:comp_finite} shows that $C_{\delta_1,p}$ is independent of the choice of $\theta$ up to symmetries of $(S,\phi)$. If $\delta_2$ is a component of $\delta$ lying on the same $\phi$-orbit as $\delta_1$, then starting the procedure using $\delta_2$ would yield the same $S_1,S_2$ and the same set of compressions up to symmetry. Thus, we write alternatively $C_{[\delta_1],p}$ for $C_{\delta_1,p}$, where $[\delta_1]$ denotes the $\phi$-orbit of components of $\delta$ containing $\delta_1$. Now, let $\mathcal C$ be the collection of all $C_{[\delta_1],p}$ obtainable from this procedure, where $[\delta_1]$ runs over $\phi$-orbits of $\pi_0(\delta)$ and $p$ runs over admissible bijections $\pi_0(S_1^+)\cong\pi_0(S_2^+)$, $S_1,S_2$ being the two pieces of $S$ adjacent to $\nu(\delta_1)$. Here, a compression of $\phi$ is presented by an orientation-reversing homeomorphism $S_1^+\to S_2^+$ as described, although one may easily transform this into a presentation using a set of compressing curves as before.

Step 2 is automatic ($\mathcal C'=\mathcal C$).

Step 3: By the observation at the beginning of the proof, we can determine which pairs $([\delta_1],p)$ are related to one another by symmetries of $(S,\phi)$, and consequently which $C_{[\delta_1],p}$'s are related to one another by symmetries. This enables us to remove duplicates in $\mathcal C'$.

\textbf{2.1.2} Step 1: For any component $\delta_1\subset\delta$ adjacent to a single piece $S_1$ of $S$, let $\Sigma_{1,1},\cdots,\Sigma_{1,M_0}$ be connected components of $S_1$, permuted cyclically by $\phi$, so that $\Sigma_{1,1}$ is adjacent to $\delta_1$ (there might be two different ways to name components of $S_1$ this way; pick one). Check whether $S_1$ is a pseudo-Anosov piece, $M_0=2M$ is even, and $\Sigma_{1,M+1}$ is adjacent to $\delta_1$. If so, define $S_1^+$ as before, with $\Sigma_{1,j}^+$'s the corresponding connected components, and proceed to find possible twisted product regions bounding $S_1^+$. We first pick $k>0$ with $\phi^{kM}$ preserving $\delta_1$ with orientation, and check if the sum of the fractional Dehn twist coefficients of $\phi^{kM}|_{S_1^+}$ about the two boundary components corresponding to $\delta_1$ is zero. If so, we apply Lemma~\ref{lem:square_root} to find all orientation-reversing homeomorphisms $\theta_1\colon\Sigma_{1,1}^+\to\Sigma_{1,M+1}^+$ rel the fixed boundaries, up to isotopy not necessarily rel boundary, with $\phi^M|_{\Sigma_{1,1}^+}$ isotopic to $\theta_1\circ\phi^M|_{\Sigma_{1,M+1}^+}\circ\theta_1$. Every such $\theta_1$ gives rise, by conjugating with iterations of $\phi$, to an orientation-reversing involution $\theta\colon S_1^+\to S_1^+$ rel (the natural involution on) the fixed boundaries exchanging $\Sigma_{1,j}^+$ and $\Sigma_{1,M+j}^+$ that commutes with $\phi|_{S_1^+}$ up to isotopy not necessarily rel boundary. Lemma~\ref{lem:FDTC_sufficient} then implies that we may modify $\theta$ by boundary Dehn twists so that it commutes with $\phi|_{S_1^+}$ up to isotopy rel the fixed boundaries. This new $\theta$ gives rise to a compression $C_\theta$ of $\phi$ as described in Section~\ref{sbsec:comp_forms}. The argument in \textbf{Case 2.1.2} in Section~\ref{sbsec:comp_finite} shows that a different choice of boundary twist modification to $\theta$ yields a compression related to $C_\theta$ by a symmetry (given by a power of the $(+1)$ multi-twist along components in $\cup_{i\ge0}\phi^i(\delta_1)$). Now, take $\mathcal C$ to be the collection of all $C_\theta$ obtainable from this procedure, for all $\delta_1$.

Step 2 is automatic. Step 3 is analogous to that in \textbf{1.3}, in that for any $C_\theta\in\mathcal C$ we can write down all compressions of $\phi$ related to $C$ by a symmetry, up to precomposing with the action of the infinite cyclic group $\Z<C(S,\phi)$ generated by the $(+1)$ multi-twist along components in $\cup_{i\ge0}\phi^i(\delta_1)$, where the component $\delta_1\subset\delta$ is taken to be the image in $S$ of a fixed boundary of the domain of $\theta$. Since it is easy to check whether two compressions of $\phi$ (especially those of the form $C_\theta$) are related by this $\Z$-action, this allows us to remove duplicates in $\mathcal C'$.

\textbf{2.2} Step 1: For any component $\delta_1\subset\delta$ adjacent to a single piece $S_1$ of $S$, let $\Sigma_{1,1},\cdots,\Sigma_{1,M}$ be connected components of $S_1$, permuted cyclically by $\phi$, so that $\Sigma_{1,1}$ is adjacent to $\delta_1$. Check whether $S_1$ is a pseudo-Anosov piece, and $\delta_1$ is adjacent to $\Sigma_{1,1}$ on both sides. If so, define $S_1^+$ as before, with $\Sigma_{1,j}^+$'s the corresponding connected components, and proceed to find possible twisted product regions bounding $S_1^+$. We first pick $k>0$ with $\phi^{kM}$ preserving $\delta_1$ with orientation, and check if the sum of the fractional Dehn twist coefficients of $\phi^{kM}|_{S_1^+}$ about the two boundary components corresponding to $\delta_1$ is zero. If so, we apply Lemma~\ref{lem:involutive_conjugator} to find all orientation-reversing involutions $\iota_1\colon\Sigma_{1,1}^+\to\Sigma_{1,1}^+$ rel the fixed boundaries, up to isotopy not necessarily rel boundary. Filter out the involutions that have fixed points (one may use Lemma~\ref{lem:conjugacy} to check if the involution is in the unique conjugacy class of free orientation-reversing involutions). Every remaining such $\iota_1$ gives rise to an orientation-reversing free involution $\iota\colon S_1^+\to S_1^+$ rel (the natural involution on) the fixed boundaries, preserving each component, that commutes with $\phi|_{S_1^+}$ up to isotopy not necessarily rel boundary. Lemma~\ref{lem:FDTC_sufficient} then implies that we may modify $\iota$ by boundary Dehn twists so that it commutes with $\phi|_{S_1^+}$ up to isotopy rel the fixed boundaries. This new $\iota$ gives rise to a compression $C_\iota$ of $\phi$ as described in Section~\ref{sbsec:comp_forms}. The argument in \textbf{Case 2.2} in Section~\ref{sbsec:comp_finite} shows that a different choice of boundary twist modification to $\iota$ yields a compression related to $C_\iota$ by a symmetry. Now, take $\mathcal C$ to be the collection of all $C_\iota$ obtainable from this procedure, for all $\delta_1$.

Step 2 is automatic. Step 3 is analogous to that in \textbf{2.1.2}.
\end{proof}

\section{Ribbon concordances of nonsimple fibered knots}\label{sec:nonsimple}
In this section, we use the classification in Section~\ref{sbsec:comp_forms} of minimal compressions of surface homeomorphisms into six canonical forms to give alternative proofs of results of Miyazaki \cite{miyazaki1994nonsimple}, avoiding the heavy use of the characteristic submanifold theory of Jaco--Shalen \cite{jaco1979} and Johannson \cite{johannson1979}.

The perspective on nonsimple fibered knots in this section uses more fibered information than the proof of Theorem~\ref{thm:finite} does. Namely, we will make use of the fact that if $K$ is a fibered knot, then every JSJ piece in its complement is fibered in the class that respects the natural boundary parametrizations, and that the monodromy of $K$ is glued from the monodromy of the JSJ pieces of its complement (with appropriate multiplicities).

\subsection{Example: The \texorpdfstring{$(2,1)$}{(2,1)}-cable of the figure \texorpdfstring{$8$}{8} knot}\label{sbsec:C21(4_1)}
Miyazaki \cite[Example~2]{miyazaki1994nonsimple} showed that $C_{2,1}(4_1)$, the $(2,1)$-cable of the figure $8$ knot, is not strongly homotopy-ribbon in any homotopy $4$-ball. This knot (and variants of it) has attracted a lot of recent attention in \cite{dai20242,kang2024cables,kang2025smooth}, which show that it is not slice, ruling out the possibility that it provides a counterexample to the slice-ribbon conjecture.

As an illustration of the classification in Section~\ref{sbsec:comp_forms}, we classify all compressions of the monodromy of $C_{2,1}(4_1)$; in particular, we will see that it does not compress to $id_{D^2}$, implying via Theorem~\ref{thm:CG} that $C_{2,1}(4_1)$ is not strongly homotopy-ribbon, giving a self-contained reproof of Miyazaki's result.

The cable pattern $C_{2,1}\subset S^1\times D^2$ is a fibered pattern with fiber a two-holed disk, denoted $D_2^2$, and monodromy $\phi_{C_{2,1}}$ freely isotopic to a half rotation of the disk, exchanging the two holes. As a rel boundary mapping class, $\phi_{C_{2,1}}$ has fractional Dehn twist coefficient $1/2$ about the boundary of the holed disk, and $\phi_{C_{2,1}}^2$ has fractional Dehn twist coefficient -$2$ about each of the inner boundaries. The figure $8$ knot $4_1$ is a fibered knot with fiber the genus one surface with one boundary $S_{1,1}$, and monodromy $\phi_{4_1}$ given by the matrix $\left(\begin{smallmatrix}2&1\\1&1\end{smallmatrix}\right)$ (as a free mapping class), with fractional Dehn twist coefficient $0$ about the boundary. The cable knot $C_{2,1}(4_1)$ is thus a fibered knot with fiber $\Sigma:=D^2_2\cup(S_{1,1}^{(1)}\sqcup S_{1,1}^{(2)})$ where $S_{1,1}^{(i)}$'s are identical copies of $S_{1,1}$, and monodromy $\phi=\phi_{C_{2,1}(4_1)}$ that restricts to $\phi_{C_{2,1}}$ on $D^2_2$, maps $S_{1,1}^{(1)}$ to $S_{1,1}^{(2)}$ via the identity map, and maps $S_{1,1}^{(2)}$ to $S_{1,1}^{(1)}$ via $\phi_{4_1}$.

The monodromy $\phi$ has one periodic piece and one pseudo-Anosov piece. The periodic piece $D^2_2$ has no essential curve, hence $\phi$ admits no minimal compression of form \textbf{1.2}. A compression of form \textbf{1.3} decreases the genus by at least $2$ on each connected component of the relevant pseudo-Anosov piece, so $\phi$ admits no minimal compression of form \textbf{1.3}. There are clearly no minimal compressions of forms \textbf{2.1.1, 2.1.2, 2.2} either. Hence, every minimal compression of $\phi$ is of the form \textbf{1.1}. The two reduction curves on $\Sigma$ are on the same $\phi$-orbit, and they cobound with $\partial\Sigma$ a planar surface; thus the compression along these two curves is non-minimal. We conclude that the unique minimal compression of $\phi$ is the compression along $\partial\Sigma$, which yields a monodromy $\phi_1\colon\Sigma_1\to\Sigma_1$, where $\Sigma_1$ is the disjoint union of $D^2$ and $S_{1,1}^{(1)}\cup S_{1,1}^{(2)}$, $\phi_1$ is the identity on $D^2$ and exchanges $S_{1,1}^{(1)}$ and $S_{1,1}^{(2)}$. Moreover, $\phi_1^2$ is freely isotopic to $\left(\begin{smallmatrix}2&1\\1&1\end{smallmatrix}\right)$ on each of $S_{1,1}^{(i)}$, and the sum of the fractional Dehn twist coefficients of $\phi_1^2$ about the reduction curve on its two sides is $-4\ne0$.

The monodromy $\phi_1$ has a single pseudo-Anosov piece, separated by a reduction curve. It is straightforward to see that $\phi_1$ admits no minimal compression of forms \textbf{1.2, 1.3, 2.1.1, 2.2}. The condition on fractional Dehn twist coefficients implies that $\phi_1$ admits no minimal compression of form \textbf{2.1.2} either. Hence, the unique minimal compression of $\phi_1$ is the form \textbf{1.1} compression given by compressing the unique reduction curve. This yields the monodromy $\phi_2\colon\Sigma_2\to\Sigma_2$, where $\Sigma_2=D^2\sqcup T^2\sqcup T^2$, $\phi_2$ is $id_{D^2}$ on $D^2$, mapping the first $T^2$ component identically onto the second and the second onto the first by $\left(\begin{smallmatrix}2&1\\1&1\end{smallmatrix}\right)$.

The monodromy $\phi_2$ admits no further minimal compressions. Thus $\phi_1$ and $\phi_2$ are all the possible nontrivial compressions of $\phi$. We note, however, that $\phi_2$ is null-cobordant (i.e. cobordant to $id_{D^2}\colon D^2\to D^2$), as $\phi_2|_{T^2\sqcup T^2}$ extends to $I\times T^2$ since $\left(\begin{smallmatrix}2&1\\1&1\end{smallmatrix}\right)$ has an orientation-reversing square root $\left(\begin{smallmatrix}1&1\\1&0\end{smallmatrix}\right)$.

\subsection{Proofs of Miyazaki's results}
We refer the reader again to \cite{budney2006jsj} for an exposition on the JSJ decomposition of knot complements.
\begin{proof}[Proof of Theorem~\ref{thm:miyazaki}]
Sufficiency is clear, as $U\le J_i\#(-J_i)$ for all $i>m$. We prove necessity by induction on the genus of $K_1\#\cdots\#K_n$. In the genus $0$ case we have $m=n=0$ and the statement is trivial. Assume the genus is positive. If $n=1$, the statement is trivial, so we also assume $n>1$. Write $K=K_1\#\cdots\#K_n$ and $J=J_1\#\cdots\#J_m$. We may assume $J\ne K$. Since the assumption and conclusion are both transitive under $\le_h$, we may further assume that there is no intermediate knot $I$ with $J\le_hI\le_hK$.

By Theorem~\ref{thm:CG}, the monodromy $\phi_K\colon F_K\to F_K$ of $K$ compresses in some compression body $C$ to the monodromy of $J$.

The knot $K$ is the splice of $K_1,\cdots,K_n$ along the keychain link $H_n$. The keychain link pattern $H_n$ is fibered, with monodromy given by the identity map on the $n$-holed disk $D^2_n$. The monodromy $\phi_K$ is given by the union of $id_{D^2_n}$ and $\phi_{K_i}\colon F_{K_i}\to F_{K_i}$, $i=1,\cdots,n$, glued along the boundary components in the standard way.

If there is a minimal subcompression $C'$ of $C$ that comes from a compression of $\phi_{K_i}$ for some $i$, we have $\partial_iC'=D^2_n\cup(F_i\sqcup(\sqcup_{j\ne i}F_{K_j}))\sqcup S$ for some connected surface $F_i$ with one boundary, and some (possibly empty or disconnected) closed surface $S$. Since $\partial_iC$ has no closed component, $S$ compresses to $\emptyset$ in $C\backslash C'$. Consequently, $\phi_{K_i}$ compresses to a monodromy on $F_i$ which is necessarily the monodromy of a (possibly trivial or composite) knot $K_i'$, implying that $K_i'\le_hK_i$. By minimality, the prime factors of $K_i'$ and the knots $K_j$ for $j\ne i$ consist of exactly the knots $\{J_1,\cdots,J_m\}$, proving the desired statement.

Assume now that there is no minimal subcompression of $C$ that comes from a compression of some $\phi_{K_i}$.

\textbf{Case 1}: Some nontrivial curve $\gamma\ne\partial F_K$ on $D^2_n\subset F_K$ compresses in $C$.

Then, compressing $\phi_K$ along $\gamma$ gives a subcompression $C'$ of $C$, with interior boundary the disjoint union of $\phi_{K_1\#\cdots\#K_r}$ and $\hat\phi_{K_{r+1}\#\cdots\#K_n}$ for some $1\le r\le n-1$ (for some reordering of the $K_i$'s), where $\hat\phi$ denotes the cap-off of a surface homeomorphism $\phi$ rel boundary. Since $\partial_iC$ has no closed components, $\hat\phi_{K_{r+1}\#\cdots\#K_n}$ compresses to $id_\emptyset$ in $C\backslash C'$. We thus have $U\le_hK_{r+1}\#\cdots\#K_n$, and hence $J=K_1\#\cdots\#K_r$ by minimality. By the induction hypothesis applied to $U\le_hK_{r+1}\#\cdots\#K_n$ and minimality, we know either $r=n-1$ or $r=n-2$ and $K_n=-K_{n-1}$, and the statement follows in both cases.

\textbf{Case 2}: No nontrivial curve $\gamma\ne\partial F_K$ on $D^2_n\subset F_K$ compresses in $C$.

By the classification of minimal compressions in Section~\ref{sbsec:comp_forms}, the only minimal subcompression $C'$ of $C$ is the compression of $\partial F_K$, which compresses $\phi_K$ to $id_{D^2}\sqcup\hat\phi_K$. We conclude that $\hat\phi_K$ compresses to $id_\emptyset$ in $C\backslash C'$ and $J=U$ is the unknot. If $n\ge3$, in view of the classification result, our assumptions imply that $\hat\phi_K$ is incompressible, a contradiction. Hence $n=2$. We claim that $K_2=-K_1$, which would finish the proof. Note that when $K_1$ (or $K_2$) is hyperbolic, the only possible compression of $\hat\phi_K$ is of the form \textbf{2.1.1} in Section~\ref{sbsec:comp_forms}, implying that the monodromies of $K_1,K_2$ are mirrors of each other, giving the claim. In the general case when $K_1,K_2$ are non-hyperbolic, we need to do some case-by-case analysis.

\textbf{Case 2.1}: The core curve in the annulus $\hat{D^2_2}\subset\hat F_K$ is not contained in the canonical reduction system of $\hat\phi_K$.

Then, each of $K_1,K_2$ is either a torus knot or a cable knot. The monodromy $\phi_{T_{p,q}}$ of a torus knot $T_{p,q}$ ($p,|q|>1$) is freely isotopic to a periodic homeomorphism on $S_{(p-1)(|q|-1),1}$ of period $p|q|$ whose quotient orbifold $D^2(p,|q|)$ is the disk with an order $p$ singularity and an order $|q|$ singularity. The monodromy $\phi_{C_{p,q}}$ of a cable pattern $C_{p,q}$ ($p>1$) is freely isotopic to a periodic homeomorphism on $S_{(p-1)(|q|-1),p+1}$ of period $p|q|$ whose quotient orbifold $A(p)$ is the annulus with an order $p$ singularity. The fractional Dehn twist coefficients of both $\phi_{T_{p,q}}$ and $\phi_{C_{p,q}}$ about the root boundary are $1/pq$. Write $K_1=C_{p_1,q_1}(K_1')$ and $K_2=C_{p_2,q_2}(K_2')$, for some (not necessarily nontrivial) fibered knots $K_1',K_2'$. By assumption, the periodic maps on two sides of $\hat{D^2_2}$ glue to a single periodic map, which implies $1/p_1q_1+1/p_2q_2=0$, i.e. $p_1q_1+p_2q_2=0$.

If both $K_1,K_2$ are torus knots, then $K_2$ being concordant to $-K_1$ implies that $K_2=-K_1$ by a classical argument. If $K_1$ is a torus knot and $K_2$ is a cable knot, then the periodic piece of $\hat F_K$ containing the annulus $\hat{D^2_2}$ quotients to the orbifold $D^2(p_1,|q_1|,p_2)$, which comes with a map $h\colon H_1(D^2(p_1,|q_1|,p_2))\to\Z/p_1|q_1|$ as explained in \textbf{1.2} in Section~\ref{sbsec:comp_forms}. By assumption, the only possible minimal compression of $\hat\phi_K$ is the compression of form \textbf{1.2} that compresses the lifts of an essential curve $\bar\gamma\subset D^2(p_1,|q_1|,p_2)$ that bounds exactly two singularities of orders $p_2,p_1$ or $p_2,|q_1|$. If $h(\bar\gamma)\ne0$, after the compression, we see a periodic component that quotients to $S^2$ with three orbifold singularities, which cannot compress to $\emptyset$, a contradiction. Thus we have $h(\bar\gamma)=0$, but then the corresponding compression violates the minimality condition imposed in \textbf{1.2} (stated more explicitly in Step 2 of \textbf{1.2} in Section~\ref{sbsec:comp_algo}), a contradiction.

Finally, we consider the case when $K_1,K_2$ are both cable knots. The periodic piece of $\hat F_K$ containing $\hat{D^2_2}$ quotients to the orbifold $A(p_1,p_2)$, where $A$ denotes the annulus. The only possible minimal compression of $\hat\phi_K$ is the compression of form \textbf{1.2} that compresses the lifts of the curve $\bar\gamma\subset A(p_1,p_2)$ that bounds a disk containing the two singularities. From the same argument as above, we find $h(\bar\gamma)=0$, and hence $p_1=p_2$, $q_1=-q_2$. The minimal compression determined by $\bar\gamma$ ``folds'' the monodromy of $C_{p_1,q_1}$ onto $C_{p_2,q_2}$, and the resulting surface homeomorphism is $(\hat\phi_{K_1'\#K_2'})^{1/p_1}$. Here, for a surface homeomorphism $\phi\colon S\to S$, by $\phi^{1/k}$ we mean the surface homeomorphism on $\sqcup_{i=1}^kS$ that maps the $j$-th copy of $S$ identically onto the $(j+1)$-th copy for $j<k$ and the last copy onto the first by $\phi$. Since compressions of $\phi^{1/k}$ are in one-to-one correspondence with those of $\phi$, we know $\hat\phi_{K_1'\#K_2'}$, like $(\hat\phi_{K_1'\#K_2'})^{1/p_1}$, compresses to $id_\emptyset$, implying that $U\le_hK_1'\#K_2'$ by Theorem~\ref{thm:CG}. By the induction hypothesis, $K_2'=-K_1'$, hence $K_2=-K_1$, as desired.

\textbf{Case 2.2}: The core curve in the annulus $\hat{D^2_2}\subset\hat F_K$ is contained in the canonical reduction system of $\hat\phi_K$.

By the classification of minimal compressions in Section~\ref{sbsec:comp_forms}, we find that the only minimal subcompression of $C$ is a compression of the form \textbf{2.1.1} that ``folds'' the (necessarily pseudo-Anosov) root piece of $\phi_{K_1}$ onto that of $\phi_{K_2}$. In particular, $K_i$ is the splice of some fibered knots $K_i^{(1)},\cdots,K_i^{(r)}$ along a hyperbolic splicing pattern $P_i$, where $P_i$ is by definition a hyperbolic knot in the complement of the $r$-component unlink $U_r$ in $S^3$, $i=1,2$, for some $r>0$ independent of $i$. By assumption, each $P_i$ is a fibered pattern with some monodromy $\phi_{P_i}\colon F_{P_i}\to F_{P_i}$, and there is an orientation-reversing homeomorphism $\theta\colon F_{P_1}\to F_{P_2}$ that intertwines $\phi_{P_1},\phi_{P_2}$ up to isotopy rel the fixed boundaries (the root boundaries). By considering parametrizations on the boundary tori as in \cite[Lemma~2.2]{miyazaki1994nonsimple}, one can show that $\theta$ actually intertwines $\phi_{P_1},\phi_{P_2}$ up to isotopy rel boundary. This implies that $P_2=-P_1$, and that the resulting surface homeomorphism after the minimal compression is given by $\sqcup_{i=1}^r\hat\phi^{1/w_i(P_1)}_{K_1^{(i)}\#K_2^{(i)}}$ (after some reordering), where $w_i(P_1)$ (necessarily nonzero) denotes the linking number between $P_1$ and the $i$-th unknot component of $U_r$ in $S^3$, taken to be positive. As before, we conclude from the induction hypothesis that $K_2^{(i)}=-K_1^{(i)}$ for all $i$, and thus $K_2=-K_1$.
\end{proof}

\begin{proof}[Proof of Corollary~\ref{cor:miyazaki}]
(1) Let $K_1,K_2$ be concordant fibered knots that are both minimal with respect to $\le_h$. By slice-ribbon, $U\le K_1\#(-K_2)$, hence $U\le_hK_1\#(-K_2)$. By Theorem~\ref{thm:miyazaki}, this implies that the collection of prime factors of $K_1$ and those of $-K_2$ (each is minimal with respect to $\le_h$ by assumption) can be grouped into mirrored pairs. By minimality, no pair can appear within the prime factors of a single $K_i$, from which we conclude $K_1=-(-K_2)=K_2$.

(2) Suppose $K$ is a fibered knot whose class in the concordance group is torsion with order $n\ge2$. We prove $n=2$. By the smooth $4$-dimensional Poincar\'e conjecture, $J\le_hK$ implies $J$ is concordant to $K$. Hence, without loss of generality, we may take $K$ to be minimal with respect to $\le_h$. By slice-ribbon, $U\le nK$, hence $U\le_hnK$, By Theorem~\ref{thm:miyazaki}, the collection of prime factors of $K$, when copied $n$ times, can be grouped into mirrored pairs. Since the collection of prime factors of $K$ itself contains no mirrored pair by minimality, each prime factor of $K$ must be (negative) amphichiral. This implies $K=-K$, so $2K=K\#(-K)$ is ribbon and hence $n=2$.
\end{proof}

\printbibliography

\end{document}